\documentclass[a4paper,twoside]{amsart}

\usepackage{verbatim}
\usepackage[british]{babel}
\usepackage{hyperref}
\usepackage{amsmath}
\usepackage{amssymb}
\usepackage{amsthm}
\usepackage{amsfonts}
\usepackage{amsxtra}
\usepackage{tabularx}
\usepackage{url}
\usepackage{fancyhdr}
\usepackage{latexsym}
\usepackage{color,graphicx}
\usepackage{pstricks}
\usepackage{pst-plot}
\usepackage{pstricks-add}
\usepackage{pst-eps}
\usepackage{epsfig}
\usepackage{multido}

\renewcommand\epsilon\varepsilon
\newcommand\cb{\mathrm{cb}}
\newcommand\diag{\operatorname{diag}}

\newcommand\Tr{\operatorname{Tr}}
\newcommand\bbR{\mathbb{R}}
\newcommand\bbC{\mathbb{C}}
\newcommand\apschur{\mathrm{AP}_{p,\mathrm{cb}}^{\mathrm{Schur}}}
\newcommand\lambdaapschur{\Lambda_{p,\mathrm{cb}}^{\mathrm{Schur}}}
\newcommand\cs{C^{\ast}}

\DeclareMathOperator\SL{SL}

\DeclareMathOperator\Sp{Sp}
\DeclareMathOperator\U{U}

\DeclareMathOperator\SO{SO}
\DeclareMathOperator\SU{SU}

\DeclareMathOperator\im{Im}

\DeclareMathOperator\id{id}

\theoremstyle{definition}
\newtheorem{thm}{Theorem}[section]
\newtheorem{dfn}[thm]{Definition}
\newtheorem{lem}[thm]{Lemma}
\newtheorem{prp}[thm]{Proposition}
\newtheorem{cor}[thm]{Corollary}

\newtheorem{rmk}[thm]{Remark}
\newtheorem{ntn}[thm]{Notation}

\author{Uffe Haagerup}
\thanks{Uf{}fe Haagerup sadly passed away on July 5, 2015.\\{}\\\indent Uf{}fe Haagerup was supported by ERC Advanced Grant no.~OAFPG 247321, the Danish Natural Science Research Council, and the Danish National Research Foundation through the Centre for Symmetry and Deformation (DNRF92).}
\address{Uffe Haagerup
\newline Department of Mathematical Sciences, University of Copenhagen
\newline Universitetsparken 5, DK-2100 Copenhagen \O, Denmark}

\author{Tim de Laat}
\thanks{Tim de Laat was supported by the Danish National Research Foundation through the Centre for Symmetry and Deformation (DNRF92).}
\address{Tim de Laat
\newline KU Leuven, Department of Mathematics
\newline Celestijnenlaan 200B -- Box 2400, B-3001 Leuven, Belgium}
\email{tim.delaat@wis.kuleuven.be}

\title{Simple Lie groups without the Approximation Property II}

\begin{document}

\begin{abstract}
We prove that the universal covering group $\widetilde{\mathrm{Sp}}(2,\mathbb{R})$ of $\mathrm{Sp}(2,\mathbb{R})$ does not have the Approximation Property (AP). Together with the fact that $\mathrm{SL}(3,\mathbb{R})$ does not have the AP, which was proved by Lafforgue and de la Salle, and the fact that $\mathrm{Sp}(2,\mathbb{R})$ does not have the AP, which was proved by the authors of this article, this finishes the description of the AP for connected simple Lie groups. Indeed, it follows that a connected simple Lie group has the AP if and only if its real rank is zero or one. By an adaptation of the methods we use to study the AP, we obtain results on approximation properties for noncommutative $L^p$-spaces associated with lattices in $\widetilde{\mathrm{Sp}}(2,\mathbb{R})$. Combining this with earlier results of Lafforgue and de la Salle and results of the second named author of this article, this gives rise to results on approximation properties of noncommutative $L^p$-spaces associated with lattices in any connected simple Lie group.
\end{abstract}

\maketitle

\section{Introduction} \label{sec:introduction}
This is the second article of the authors on the Approximation Property (AP) for Lie groups. In the first article on this topic, the authors proved that $\mathrm{Sp}(2,\mathbb{R})$ does not satisfy the AP \cite{haagerupdelaat1}. Together with the earlier established fact that $\mathrm{SL}(3,\mathbb{R})$ does not have the AP, which was proved by Lafforgue and de la Salle in \cite{ldls}, this implied that if $G$ is a connected simple Lie group with finite center and real rank greater than or equal to two, then $G$ does not satisfy the AP. In \cite{haagerupdelaat1}, it was pointed out that in order to extend this result to the class of connected simple Lie groups with real rank greater than or equal to two, i.e., not necessarily with finite center, it would be sufficient to prove that the universal covering group $\widetilde{\mathrm{Sp}}(2,\mathbb{R})$ of $\mathrm{Sp}(2,\mathbb{R})$ does not satisfy the AP. The main goal of this article is to prove this. This finishes the description of the AP for connected simple Lie groups. Indeed, it follows that a connected simple Lie group has the AP if and only if its real rank is zero or one.

In this article we are mainly interested in Lie groups, but many definitions are given in the setting of locally compact groups. We always assume locally compact groups to be second countable and Hausdorff. Before we state the main results of this article, we give some background (see Section 1 of \cite{haagerupdelaat1} for a more extensive account of the background).

Let $G$ be a locally compact group. Denote by $A(G)$ its Fourier algebra and by $M_0A(G)$ the space of completely bounded Fourier multipliers on $G$. Recall that $G$ is said to have the Approximation Property for groups (AP) if there is a net $(\varphi_{\alpha})$ in the Fourier algebra $A(G)$ such that $\varphi_{\alpha} \to 1$ in the $\sigma(M_0A(G),M_0A(G)_*)$-topology, where $M_0A(G)_*$ denotes the natural predual of $M_0A(G)$, as introduced in \cite{decannierehaagerup}.

The AP was defined by the first named author and Kraus in \cite{haagerupkraus} as a version for groups of the Banach space approximation property (BSAP) of Grothendieck. To see the connection, recall first that Banach spaces have a natural noncommutative analogue, namely, operator spaces. Recall that an operator space $E$ is a closed linear subspace of the bounded operators $\mathcal{B}(\mathcal{H})$ on a Hilbert space $\mathcal{H}$. Operator spaces have a remarkably rich structure (see \cite{effrosruanoperatorspaces}, \cite{pisieroperatorspaces}). For the class of operator spaces, which contains the class of $\cs$-algebras, a well-known version of the BSAP is known, namely, the operator space approximation property (OAP). The first named author and Kraus proved that a discrete group $\Gamma$ has the AP if and only if its reduced $\cs$-algebra $C^{\ast}_{\lambda}(\Gamma)$ has the OAP.

The AP also relates to other approximation properties for groups (see \cite{brownozawa} for an extensive text on approximation properties for groups and operator algebras). It is known that weak amenability (which is strictly weaker than amenability) strictly implies the AP. Amenability and weak amenability have been studied thoroughly for Lie groups. Indeed, a connected simple Lie group with real rank zero is amenable and a connected simple Lie group with real rank one is weakly amenable (see \cite{cowlinghaagerup} and \cite{hansen}). Also, it has been known for some time that connected simple Lie groups with real rank greater than or equal to two are not weakly amenable (see \cite{haagerupgroupcsacbap} and \cite{dorofaeff}). In addition, weak amenability was studied for a larger class of connected Lie groups in \cite{cowlingdorofaeffseegerwright}. The AP has been less studied than weak amenability. In particular, until the work of Lafforgue and de la Salle, no example of an exact group without the AP was known.

The key theorem of this article is as follows.
\newtheorem*{thm:covsp2noap}{Theorem \ref{thm:covsp2noap}}
\begin{thm:covsp2noap}
  The universal covering group $\widetilde{\Sp}(2,\bbR)$ of the symplectic group $\Sp(2,\bbR)$ does not have the Approximation Property.
\end{thm:covsp2noap}
Combining this with the fact that $\SL(3,\mathbb{R})$ does not have the AP, as established by Lafforgue and de la Salle, and the fact that $\mathrm{Sp}(2,\mathbb{R})$ does not have the AP, as proved by the authors, the following main result follows.
\newtheorem*{thm:ctdsimple}{Theorem \ref{thm:apctdsimple}}
\begin{thm:ctdsimple}
  Let $G$ be a connected simple Lie group. Then $G$ has the Approximation Property if and only if $G$ has real rank zero or one.
\end{thm:ctdsimple}
There are important differences between the approach of Lafforgue and de la Salle for the proof of the fact that $\mathrm{SL}(3,\mathbb{R})$ does not have the AP in \cite{ldls} and the approach of the authors for proving the failure of the AP for $\mathrm{Sp}(2,\mathbb{R})$ in \cite{haagerupdelaat1} and for its universal covering group in this article. Indeed, the method of Lafforgue and de la Salle gives information about approximation properties for certain noncommutative $L^p$-spaces associated with lattices in $\mathrm{SL}(3,\mathbb{R})$, which the method of the authors does not. However, the latter is more direct, since it suffices to consider completely bounded Fourier multipliers rather than completely bounded multipliers on Schatten classes.

Noncommutative $L^p$-spaces are important examples of the earlier mentioned operator spaces. Let $M$ be a finite von Neumann algebra with normal faithful trace $\tau$. For $1 \leq p < \infty$, the noncommutative $L^p$-space $L^p(M,\tau)$ is defined as the completion of $M$ with respect to the norm $\|x\|_p=\tau((x^*x)^{\frac{p}{2}})^{\frac{1}{p}}$, and for $p=\infty$, we put $L^{\infty}(M,\tau)=M$ (with operator norm). Noncommutative $L^p$-spaces can be realized by interpolating between $M$ and $L^1(M,\tau)$ (see \cite{kosaki}). This leads to an operator space structure on them (see \cite{pisieroh},\cite{jungeruan}).

An operator space $E$ is said to have the completely bounded approximation property (CBAP) if there exists a net $(F_{\alpha})$ of finite-rank maps on $E$ with $\sup_{\alpha}\|F_{\alpha}\|_{cb} < C$ for some $C > 0$ such that $\lim_{\alpha} \|F_{\alpha}x-x\|=0$ for every $x \in E$. The infimum of all possible $C$'s is denoted by $\Lambda(E)$. If $\Lambda(E)=1$, then $E$ has the completely contractive approximation property (CCAP). An operator space $E$ is said to have the operator space approximation property (OAP) if there exists a net $(F_{\alpha})$ of finite-rank maps on $E$ such that $\lim_{\alpha} \|(\id_{\mathcal{K}(\ell^2)} \otimes F_{\alpha})x-x\|=0$ for all $x \in \mathcal{K}(\ell^2) \otimes_{\min} E$. Here, $\mathcal{K}(\ell^2)$ denotes the space of compact operators on $\ell^2$. The CBAP goes back to \cite{decannierehaagerup}, and the OAP was defined in \cite{effrosruanap}. By definition, the CCAP implies the CBAP, which in turn implies the OAP.

It was shown by Junge and Ruan \cite{jungeruan} that if $\Gamma$ is a weakly amenable countable discrete group (resp.~a countable discrete group with the AP), and if $p \in (1,\infty)$, then $L^p(L(\Gamma))$ has the CBAP (resp.~the OAP), where $L(\Gamma)$ denotes the group von Neumann algebra of $\Gamma$. The method of Lafforgue and de la Salle can be used to prove the failure of the CBAP and OAP for noncommutative $L^p$-spaces. The key ingredient of their method is the property of completely bounded approximation by Schur multipliers on $S^p$, denoted $\apschur$, which is weaker than the AP for $p \in (1,\infty)$. Indeed, they prove that if $p \in (1,\infty)$ and $\Gamma$ is a countable discrete group with the AP, then $\Lambda_{p,\mathrm{cb}}^{\mathrm{Schur}}(\Gamma)=1$ (see \cite[Corollary 3.12]{ldls}). Also, they prove that if $p \in (1,\infty)$ and $\Gamma$ is a countable discrete group such that $L^p(L(\Gamma))$ has the OAP, then 
$\Lambda_{p,\mathrm{cb}}^{\mathrm{Schur}}(\Gamma)=1$ (see \cite[Corollary 3.13]{ldls}). Using this, they prove that for $p \in [1,\frac{4}{3}) \cup (4,\infty]$ and a lattice $\Gamma$ in $\mathrm{SL}(3,\mathbb{R})$, the noncommutative $L^p$-space $L^p(L(\Gamma))$ does not have the OAP or CBAP.

In \cite{delaat1}, the second named author generalized the results of Lafforgue and de la Salle on approximation properties for noncommutative $L^p$-spaces associated with lattices in $\mathrm{SL}(3,\mathbb{R})$ to noncommutative $L^p$-spaces associated with lattices in connected simple Lie groups with finite center and real rank greater than or equal to two. In this article, we will in turn generalize these results to connected simple Lie groups with real rank greater than or equal to two that do not necessarily have finite center, as is illustrated by our main result on noncommutative $L^p$-spaces.
\newtheorem*{thm:nclps}{Theorem \ref{thm:nclpsmain}}
\begin{thm:nclps}
  Let $\Gamma$ be a lattice in a connected simple Lie group with real rank greater than or equal to two. For $p \in [1,\frac{12}{11}) \cup (12,\infty]$, the noncommutative $L^p$-space $L^p(L(\Gamma))$ does not have the OAP or CBAP.
\end{thm:nclps}
It may very well be possible that the range of $p$-values for which the CBAP and OAP fail is larger than $[1,\frac{12}{11}) \cup (12,\infty]$. We will comment on this in further detail in Section \ref{sec:mainresults}.

This article is organized as follows. In Section \ref{sec:preliminaries}, we recall some preliminaries. In Section \ref{sec:covsp2r}, we prove that $\widetilde{\mathrm{Sp}}(2,\mathbb{R})$ does not have the AP. We prove the results on noncommutative $L^p$-spaces in Section \ref{sec:nclpspaces}. The results will be summarized and combined to our general results in Section \ref{sec:mainresults}. Appendix \ref{sec:sgp} gives a connection between spherical functions for Gelfand pairs and their analogues for strong Gelfand pairs that might give a deeper understanding of certain results that are proved in Section \ref{sec:covsp2r}. The material in that appendix follows from discussions of the second named author with Thomas Danielsen. This material might be known to experts, but we could not find an explicit reference. 

\section{Preliminaries} \label{sec:preliminaries}
\subsection{Universal covering groups} \label{subsec:universalcoveringgroups}
Let $G$ be a connected Lie group. A covering group of $G$ is a Lie group $\widetilde{G}$ with a surjective Lie group homomorphism $\sigma:\widetilde{G} \rightarrow G$, in such a way that $(\widetilde{G},\sigma)$ is a covering space of $G$ (in the topological sense). A simply connected covering space is called a universal covering space. Every connected Lie group $G$ has a universal covering space $\widetilde{G}$. Let $\sigma:\widetilde{G} \rightarrow G$ be the corresponding covering map, and let $\tilde{1} \in \sigma^{-1}(1)$. Then there exists a unique multiplication on $\widetilde{G}$ that makes $\widetilde{G}$ into a Lie group in such a way that $\sigma$ is a surjective Lie group homomorphism. The group $\widetilde{G}$ is called a universal covering group of the Lie group $G$. Universal covering groups of connected Lie groups are unique up to isomorphism. They also satisfy the exact sequence $1 \rightarrow \pi_1(G) \rightarrow \widetilde{G} \rightarrow G \rightarrow 1$, where $\pi_1(G)$ denotes the fundamental group of $G$. For details on universal covering groups, see \cite[Section I.11]{knapp}.

\subsection{Polar decomposition of Lie groups} \label{subsec:kakdecomposition}
Every connected semisimple Lie group $G$ has a polar decomposition $G=KAK$, where $K$ arises from a Cartan decomposition $\mathfrak{g}=\mathfrak{k} + \mathfrak{p}$ (the group $K$ has Lie algebra $\mathfrak{k}$), and $A$ is an abelian Lie group such that its Lie algebra $\mathfrak{a}$ is a maximal abelian subspace of $\mathfrak{p}$. If $G$ has finite center, then $K$ is a maximal compact subgroup. The dimension of the Lie algebra $\mathfrak{a}$ of $A$ is called the real rank of $G$ and is denoted by $\mathrm{rank}_{\mathbb{R}}(G)$. In general, given a polar decomposition $G=KAK$, it is not the case that for $g \in G$ there exist unique $k_1,k_2 \in K$ and $a \in A$ such that $g=k_1ak_2$. However, after choosing a set of positive roots and restricting to the closure $\overline{A^{+}}$ of the positive Weyl chamber $A^{+}$, we still have $G=K\overline{A^{+}}K$. Moreover, if $g=k_1ak_2$, where $k_1,k_2 \in K$ and $a \in \overline{A^{+}}$, then $a$ is unique. Note that we can choose any Weyl chamber to be the positive one by choosing the set of positive roots correspondingly. We also use the terminology polar decomposition for such a $K\overline{A^{+}}K$ decomposition. For details, see \cite[Section IX.1]{helgasonlie}.

\subsection{Gelfand pairs and spherical functions} \label{subsec:gelfandpairs}
Let $G$ be a locally compact group (with Haar measure $dg$) with a compact subgroup $K$ (with normalized Haar measure $dk$). A function $\varphi:G \rightarrow \mathbb{C}$ is said to be $K$-bi-invariant if for all $g \in G$ and $k_1,k_2 \in K$, we have $\varphi(k_1gk_2)=\varphi(g)$. We denote the space of continuous $K$-bi-invariant compactly supported functions by $C_c(K \backslash G \slash K)$. If the subalgebra $C_c(K \backslash G \slash K)$ of the (convolution) algebra $C_c(G)$ is commutative, then the pair $(G,K)$ is called a Gelfand pair. Equivalently, the pair $(G,K)$ is a Gelfand pair if and only if for every irreducible unitary representation $\pi$ on a Hilbert space $\mathcal{H}$, the space $\mathcal{H}_e=\{ \xi \in \mathcal{H} \mid \forall k \in K:\,\pi(k)\xi=\xi \}$ consisting of $K$-invariant vectors is at most one-dimensional. For a Gelfand pair $(G,K)$, a function $h \in C(K \backslash G \slash K)$ is called spherical if the functional $\chi$ on $C_c(K \backslash G \slash K)$ given by $\chi(\varphi)=\int_G \varphi(g)h(g^{-1})dg$ for $\varphi \in C_c(K \backslash G \slash K)$ defines a nontrivial character. The theory of Gelfand pairs and spherical functions is well-established and goes back to Gelfand \cite{gelfand}. For more recent accounts of the theory, we refer the reader to \cite{vandijk}, \cite{faraut}, \cite{wolf}. 

Let $G$ be a locally compact group with closed subgroup $H$. A function $\varphi:G \rightarrow \mathbb{C}$ is said to be $\mathrm{Int}(H)$-invariant if $\varphi(hgh^{-1})=\varphi(g)$ for all $g \in G$ and $h \in H$. The space of continuous $\mathrm{Int}(H)$-invariant functions is denoted by $C(G \slash \slash H)$.

Let now $G$ be a locally compact group with compact subgroup $K$. The pair $(G,K)$ is called a strong Gelfand pair if the subalgebra $C_c(G \slash\slash K)$ of $C_c(G)$ is commutative. In the setting of locally compact groups, the notion of strong Gelfand pair goes back to Goldrich and Wigner \cite{goldrichwigner}. It is well-known that whenever $G$ is a locally compact group with a compact subgroup $K$, then $(G,K)$ is a strong Gelfand pair if and only if $(G \times K,\Delta K)$ (where $\Delta K$ is the diagonal subgroup) is a Gelfand pair.

It turns out that certain results of Section \ref{sec:covsp2r} can be understood on a deeper level in the setting of strong Gelfand pairs, in particular when one considers the analogue of spherical functions in this setting. This is discussed in Appendix \ref{sec:sgp}. The analogues of spherical functions already occurred in \cite{godement}. 

\subsection{The Fourier algebra}
Let $G$ be a locally compact group. The Fourier algebra $A(G)$ is defined as the space consisting of the coefficients of the left-regular representation $\lambda:G \rightarrow \mathcal{B}(L^2(G))$. It was introduced by Eymard \cite{eymard} (see also \cite{eymard2}). More precisely, $\varphi \in A(G)$ if and only if there exist $\xi,\eta \in L^2(G)$ such that for all $g \in G$, we have $\varphi(g)=\langle \lambda(g)\xi,\eta \rangle$. The Fourier algebra $A(G)$ is a Banach space with respect to the norm defined by $\|\varphi\|_{A(G)}=\min \{ \|\xi\|\|\eta\| \mid \forall g \in G \; \varphi(g)=\langle \lambda(g)\xi,\eta \rangle \}$. We have $\|\varphi\|_{\infty} \leq \|\varphi\|_{A(G)}$ for all $\varphi \in A(G)$, and $A(G)$ is $\|.\|_{\infty}$-dense in $C_0(G)$. Eymard showed that $A(G)$ can be identified isometrically with the predual of the group von Neumann algebra $L(G)$ of $G$.

\subsection{Completely bounded Fourier multipliers on compact Gelfand pairs} \label{subsec:cbfmcgp}
A function $\varphi:G \rightarrow \mathbb{C}$ is said to be a Fourier multiplier if and only if $\varphi\psi \in A(G)$ for all $\psi \in A(G)$. Let $MA(G)$ denote the Banach space of multipliers of $A(G)$ equipped with the norm given by $\|\varphi\|_{MA(G)}=\|m_{\varphi}\|$, where $m_{\varphi}:A(G) \rightarrow A(G)$ denotes the associated multiplication operator. A multiplier $\varphi$ is said to be completely bounded if the operator $M_{\varphi}:L(G) \rightarrow L(G)$ induced by $m_{\varphi}$ is completely bounded. The space of completely bounded multipliers is denoted by $M_0A(G)$, and with the norm $\|\varphi\|_{M_0A(G)}=\|M_{\varphi}\|_{\cb}$, it forms a Banach space. It is known that $A(G) \subset M_0A(G) \subset MA(G)$.

It was proved by Bo\.zejko and Fendler in \cite{bozejkofendler1} that $\varphi \in M_0A(G)$ if and only if there exist bounded continuous maps $P,Q:G \rightarrow \mathcal{H}$, where $\mathcal{H}$ is a Hilbert space, such that $\varphi(g_2^{-1}g_1)=\langle P(g_1),Q(g_2) \rangle$ for all $g_1,g_2 \in G$. Here $\langle .,. \rangle$ denotes the inner product on $\mathcal{H}$. In this characterization, $\|\varphi\|_{M_0A(G)}=\min\{\|P\|_{\infty}\|Q\|_{\infty}\}$, where the minimum is taken over all possible pairs $(P,Q)$ for which $\varphi(g_2^{-1}g_1)=\langle P(g_1),Q(g_2) \rangle$ for all $g_1,g_2 \in G$.

Suppose now that $(G,K)$ is a compact Gelfand pair, i.e., the group $G$ is compact and $(G,K)$ is a Gelfand pair. Then for every irreducible representation $\pi$ on $\mathcal{H}$, the space $\mathcal{H}_e$ as defined in Section \ref{subsec:gelfandpairs} is at most one-dimensional. Let $P_{\pi}=\int_K \pi(k)dk$ denote the projection onto $\mathcal{H}_e$, and set $\hat{G}_K=\{ \pi \in \hat{G} \mid P_{\pi} \neq 0 \}$, where $\hat{G}$ denotes the unitary dual of $G$. We proved the following result in \cite[Proposition 2.3]{haagerupdelaat1}.
\begin{prp} \label{prp:cbfmcgp}
  Let $(G,K)$ be a compact Gelfand pair, and let $\varphi$ be a $K$-bi-invariant completely bounded Fourier multiplier. Then $\varphi$ has a unique decomposition $\varphi(g)=\sum_{\pi \in \hat{G}_K} c_{\pi}h_{\pi}(g)$ for all $g \in G$, where $h_{\pi}(g)=\langle \pi(g)\xi_{\pi},\xi_{\pi} \rangle$ is the positive definite spherical function associated with the representation $\pi$ with $K$-invariant cyclic vector $\xi_{\pi}$, and $\sum_{\pi \in \hat{G}_K} |c_{\pi}|=\|\varphi\|_{M_0A(G)}$.
\end{prp}

\subsection{The Approximation Property} \label{subsec:ap}
We recall the definition and basic properties of the Approximation Property for groups (AP), as introduced by the first named author and Kraus \cite{haagerupkraus}.
\begin{dfn}
A locally compact group $G$ is said to have the Approximation Property for groups (AP) if there is a net $(\varphi_{\alpha})$ in $A(G)$ such that $\varphi_{\alpha} \to 1$ in the $\sigma(M_0A(G),M_0A(G)_*)$-topology, where $M_0A(G)_*$ denotes the natural predual of $M_0A(G)$ as introduced in \cite{decannierehaagerup} (see also \cite{haagerupkraus} and \cite{haagerupdelaat1}).
\end{dfn}
It was proved by the first named author and Kraus that if $G$ is a locally compact group and $\Gamma$ is a lattice in $G$, then $G$ has the AP if and only if $\Gamma$ has the AP \cite[Theorem 2.4]{haagerupkraus}. The AP passes to closed subgroups, as is proved in \cite[Proposition 1.14]{haagerupkraus}. Also, if $H$ is a closed normal subgroup of a locally compact group $G$ such that both $H$ and $G \slash H$ have the AP, then $G$ has the AP \cite[Theorem 1.15]{haagerupkraus}. Moreover, if $G_1$ and $G_2$ are two locally isomorphic connected simple Lie groups with finite center such that $G_1$ has the AP, then $G_2$ has the AP \cite[Proposition 2.4]{haagerupdelaat1}.

\subsection{Preliminaries for the results on noncommutative $L^p$-spaces} \label{subsec:nclps}
These preliminaries are only relevant for Section \ref{sec:nclpspaces}. For a more extensive account, we refer to \cite{ldls}, \cite{delaat1}. 

\subsubsection{Schur multipliers on Schatten classes}
For $p \in [1,\infty]$ and a Hilbert space $\mathcal{H}$, let $S^p(\mathcal{H})$ denote the $p^{\textrm{th}}$ Schatten class on $\mathcal{H}$. We identify $S^2(\mathcal{H})$ with $\mathcal{H}^* \otimes \mathcal{H}$, and for a $\sigma$-finite measure space $(X,\mu)$, we identify $L^2(X,\mu)^*$ with $L^2(X,\mu)$ by the duality bracket $\langle f,g \rangle=\int_X fg d\mu$. It follows that $S^2(L^2(X,\mu))$ can be identified with $L^2(X \times X,\mu \otimes \mu)$. Hence, every Schur multiplier on $S^2(L^2(X,\mu))$ comes from a function $\psi \in L^{\infty}(X \times X,\mu \otimes \mu)$ acting by multiplication on $L^2(X \times X,\mu \otimes \mu)$.
\begin{dfn} \label{dfn:multiplier}
  Let $p \in [1,\infty]$, and let $\psi \in L^{\infty}(X \times X,\mu \otimes \mu)$. The Schur multiplier with symbol $\psi$ is said to be bounded (resp.~completely bounded) on $S^p(L^2(X,\mu))$ if it maps $S^p(L^2(X,\mu)) \cap S^2(L^2(X,\mu))$ into $S^p(L^2(X,\mu))$ by $T_k \mapsto T_{\psi k}$ (where $T_k$ denotes the integral operator with kernel $k$), and if this map extends (necessarily uniquely) to a bounded (resp.~completely bounded) map $M_{\psi}$ on $S^p(L^2(X,\mu))$.
\end{dfn}
The norm of a bounded multiplier $\psi$ is defined by $\|\psi\|_{MS^p(L^2(X,\mu))}=\|M_{\psi}\|$, and its completely bounded norm by $\|\psi\|_{cbMS^p(L^2(X,\mu))}=\|M_{\psi}\|_{cb}$. The spaces of multipliers and completely bounded multipliers are denoted by $MS^p(L^2(X,\mu))$ and $cbMS^p(L^2(X,\mu))$, respectively. It follows that for every $p \in [1,\infty]$ and $\psi \in L^{\infty}(X \times X,\mu \otimes \mu)$, we have $\|\psi\|_{\infty} \leq \|\psi\|_{MS^p(L^2(X,\mu))} \leq \|\psi\|_{cbMS^p(L^2(X,\mu))}$.

\subsubsection{Schur multipliers on compact Gelfand pairs}
In this section, we recall results from \cite[Section 2]{delaat1} that are analogues in the setting of multipliers on Schatten classes of the results of Section \ref{subsec:cbfmcgp}. For proofs, we refer to \cite{delaat1}.

For a locally compact group $G$ and a function $\varphi \in L^{\infty}(G)$, we define the function $\check{\varphi} \in L^{\infty}(G \times G)$ by $\check{\varphi}(g_1,g_2)=\varphi(g_1^{-1}g_2)$.

In what follows, let $G$ and $K$ be Lie groups such that $(G,K)$ is a compact Gelfand pair. Let $X=G \slash K$ denote the homogeneous space corresponding with the canonical transitive action of $G$. The group $K$ is the stabilizer subgroup of a certain element $e_0 \in X$. It follows that $L^2(X)=\oplus_{\pi \in \hat{G}_K} \mathcal{H}_{\pi}$. Let $h_{\pi}$ denote the spherical function corresponding to the equivalence class $\pi$ of representations. Then for every $\varphi \in L^2(K \backslash G \slash K)$ we have $\varphi=\sum_{\pi \in \hat{G}_K} c_{\pi} \dim{\mathcal{H}_{\pi}} h_{\pi}$, where $c_{\pi}=\langle \varphi,h_{\pi} \rangle$. It also follows that for any $\varphi \in C(K \backslash G \slash K)$, there exists a continuous function $\psi:X \times X \rightarrow \bbC$ such that for all $g_1,g_2 \in G$, we have $\varphi(g_1^{-1}g_2)=\psi(g_1e_0,g_2e_0)$. Let $\varphi:G \rightarrow \bbC$ be a continuous $K$-bi-invariant function such that $\check{\varphi} \in cbMS^p(L^2(G))$ for some $p \in [1,\infty]$. Then $\|\psi\|_{cbMS^p(L^2(X))}=\|\check{\varphi}\|_{cbMS^p(L^2(G))}$, where $\psi:X \times X \rightarrow \bbC$ is as defined above. If $K$ is an infinite group, then these norms are equal to $\|\check{\varphi}\|_{MS^p(L^2(G))}$.

Let $(G,K)$ be a compact Gelfand pair, let $p \in [1,\infty)$, and let $\varphi:G \rightarrow \bbC$ be a continuous $K$-bi-invariant function such that $\check{\varphi} \in MS^p(L^2(G))$. Then $\left( \sum_{\pi \in \hat{G}_K} |c_{\pi}|^p (\dim{\mathcal{H}_{\pi}}) \right)^{\frac{1}{p}} \leq \|\check{\varphi}\|_{MS^p(L^2(G))}$, where $c_{\pi}$ and $\mathcal{H}_{\pi}$ are as before.

\subsubsection{The $\apschur$}
The $\apschur$ was defined in \cite{ldls}. Its relevance to us, including certain important properties, was described in Section \ref{sec:introduction}.
\begin{dfn}(see \cite[Definition 2.2]{ldls}) Let $G$ be a locally compact group, and let $1 \leq p \leq \infty$. The group $G$ is said to have the property of completely bounded approximation by Schur multipliers on $S^p$, denoted $\apschur$, if there exists a constant $C > 0$ and a net $(\varphi_{\alpha})$ in $A(G)$ such that $\varphi_{\alpha} \to 1$ uniformly on compacta and $\sup_{\alpha} \|\check{\varphi}_{\alpha}\|_{cbMS^p(L^2(G))} \leq C$. The infimum of these $C$'s is denoted by $\lambdaapschur (G)$.
\end{dfn}
It was proved by Lafforgue and de la Salle that if $G$ is a locally compact group and $\Gamma$ is a lattice in $G$, then for $1 \leq p \leq \infty$, we have $\Lambda_{p,\mathrm{cb}}^{\mathrm{Schur}}(\Gamma)=\Lambda_{p,\mathrm{cb}}^{\mathrm{Schur}}(G)$ (see \cite[Theorem 2.5]{ldls}). More properties of the $\apschur$ are discussed in \cite{ldls} and \cite{delaat1}.

\section{The group $\widetilde{\mathrm{Sp}}(2,\mathbb{R})$ does not have the AP} \label{sec:covsp2r}
In this section, we prove that the universal covering group $\widetilde{\mathrm{Sp}}(2,\mathbb{R})$ of $\mathrm{Sp}(2,\mathbb{R})$ does not have the AP. Hereto, let us first recall the definition of $\mathrm{Sp}(2,\mathbb{R})$ and describe a realization of $\widetilde{\Sp}(2,\bbR)$.

Let $I_2$ denote the $2 \times 2$ identity matrix, and let the matrix $J$ be defined by
\[
  J=\left( \begin{array}{cc} 0 & I_2 \\ -I_2 & 0 \end{array} \right).
\]
Recall that the symplectic group $\mathrm{Sp}(2,\mathbb{R})$ is defined as the Lie group
\[
	\mathrm{Sp}(2,\mathbb{R}):=\{g \in \mathrm{GL}(4,\mathbb{R}) \mid g^T J g = J\}.
\]
Here, $g^T$ denotes the transpose of $g$. Let $K$ denote the maximal compact subgroup of $\Sp(2,\bbR)$ given by
\[
  K= \bigg\{ \left( \begin{array}{cc} A & -B \\ B & A \end{array} \right) \in \mathrm{M}_{4}(\mathbb{R}) \biggm\vert A+iB \in \mathrm{U}(2) \bigg\}.
\]
This group is isomorphic to $\mathrm{U}(2)$. A polar decomposition of $\Sp(2,\bbR)$ is given by $\Sp(2,\bbR)=K\overline{A^{+}}K$, where
\[
	\overline{A^{+}}=\left\{D(\beta,\gamma)= \left( \begin{array}{cccc} e^{\beta} & 0 & 0 & 0 \\ 0 & e^{\gamma} & 0 & 0 \\ 0 & 0 & e^{-\beta} & 0 \\ 0 & 0 & 0 & e^{-\gamma} \end{array} \right) \Biggm\vert \beta \geq \gamma \geq 0\right\}.
\]
Different explicit realizations of $\widetilde{\mathrm{Sp}}(2,\mathbb{R})$ can be found in the literature. An incomplete list is given by \cite{lionvergne}, \cite{rawnsley}, \cite{wolfcovering}. We use the realization in terms of circle functions, given recently by Rawnsley \cite{rawnsley}, and in what follows we use of some of his computations. In fact, he describes a method that gives a realization of the universal covering group of any connected Lie group $G$ with fundamental group $\pi_1(G)$ isomorphic to $\mathbb{Z}$ admitting a so-called (normalized) circle function. Firstly, we briefly describe Rawnsley's general construction.

Let $G$ be a connected Lie group with $\pi_1(G) \cong \mathbb{Z}$. A circle function on $G$ is a smooth function $c:G \rightarrow \mathbb{T}$, where $\mathbb{T}$ denotes the circle (as a subspace of $\mathbb{C}$), that induces an isomorphism of the fundamental groups of $G$ and $\mathbb{T}$. Such a function is said to be normalized if $c(1)=1$ and $c(g^{-1})=c(g)^{-1}$. If $G$ admits a circle function, it admits one and only one normalized circle function.

Let $G$ be a connected Lie group with fundamental group isomorphic to $\mathbb{Z}$ that admits a normalized circle function. Then there exists a unique smooth function $\eta:G \times G \rightarrow \mathbb{R}$ such that
\[
  c(g_1g_2)=c(g_1)c(g_2)e^{i\eta(g_1,g_2)}
\]
for all $g_1,g_2 \in G$ and $\eta(1,1)=0$. Furthermore, it follows that $\eta(g,1)=\eta(1,g)=\eta(g,g^{-1})=0$ and $\eta(g_1,g_2)+\eta(g_1g_2,g_3)=\eta(g_1,g_2g_3)+\eta(g_2,g_3)$ for all $g \in G$ and $g_1,g_2,g_3 \in G$.

Let $G$ be a connected Lie group with normalized circle function $c$, and let
\begin{equation} \label{eq:uncov}
  \widetilde{G}=\{(g,t) \in G \times \mathbb{R} \mid c(g)=e^{it}\}.
\end{equation}
The space $\widetilde{G}$ is a smooth manifold of the same dimension as $G$. A multiplication on $\widetilde{G}$ is given by
\[
  (g_1,t_1)(g_2,t_2)=(g_1g_2,t_1+t_2+\eta(g_1,g_2)).
\]
With this multiplication, $\widetilde{G}$ is a Lie group with identity $\tilde{1}=(1,0)$, where $1$ denotes the identity element of $G$, and inverse given by $(g,t)^{-1}=(g^{-1},-t)$. The map $\sigma:\widetilde{G} \rightarrow G$, $(g,t) \mapsto g$ (with kernel $\{(1,2\pi k) \in G \times \mathbb{R} \mid k \in \mathbb{Z}\}$) defines a universal covering map from $\widetilde{G}$ onto $G$.\\

{\bf In the rest of this section, let $G=\Sp(2,\bbR)$ and $\widetilde{G}=\widetilde{\mathrm{Sp}}(2,\mathbb{R})$.}\\

We now give the explicit functions $c$ and $\eta$ for $\widetilde{\mathrm{Sp}}(2,\mathbb{R})$. Let $M_4(\bbR)_0$ denote the subspace of $M_4(\bbR)$ given by
\[
  M_4(\bbR)_0=\bigg\{ \left( \begin{array}{cc} A & -B \\ B & A \end{array} \right) \biggm\vert A,B \in M_2(\bbR) \bigg\},
\]
and let $\iota:M_4(\bbR)_0 \rightarrow M_2(\bbC)$ be given by
\[
  \iota:\left( \begin{array}{cc} A & -B \\ B & A \end{array} \right) \mapsto A+iB.
\]
The map $\iota$ is an algebra homomorphism. For an element $g \in G$, let $C_g=\frac{1}{2}(g+(g^T)^{-1})$ and $D_g=\frac{1}{2}(g-(g^T)^{-1})$. Note that $g=C_g+D_g$. As described by Rawnsley, the connected Lie group $G$ admits a normalized circle function; namely, the function $c:G \rightarrow \mathbb{T}$ given by
\begin{equation} \label{eq:circlefunction}
  c(g)=\frac{\det(\iota(C_g))}{|\det(\iota(C_g))|}.
\end{equation}
With this circle function, the manifold $\widetilde{G}$ is given through \eqref{eq:uncov}. Let $Z_g=C_g^{-1}D_g$. The function $\eta$ (which is needed to define the multiplication on $\widetilde{G}$) corresponding to the circle function $c$ is given by
\[
  \eta(g_1,g_2)=\im(\Tr(\iota(\log(1-Z_{g_1}Z_{g_2^{-1}}))))=\im(\Tr(\iota(\log(C_{g_1}^{-1}C_{g_1g_2}C_{g_2}^{-1})))).
\]
The logarithm is well-defined, since $\|Z_{g_1}Z_{g_2^{-1}}\| < 1$ (see \cite[Section 4]{rawnsley}). It was also proved by Rawnsley that $|\eta(g_1,g_2)| < \pi$ for all $g_1,g_2 \in G$ (see \cite[Lemma 14]{rawnsley}).
\begin{rmk}
  Everything that we described so far for $G$ can be generalized to $\Sp(n,\bbR)$ for $n \geq 1$ (see \cite{rawnsley}).
\end{rmk}
The rest of this section is devoted to proving the following theorem.
\begin{thm} \label{thm:covsp2noap}
	The group $\widetilde{G}=\widetilde{\Sp}(2,\mathbb{R})$ does not have the AP.
\end{thm}
Firstly, we elaborate on the structure of $\widetilde{G}$. Let $\mathfrak{g}$ denote the Lie algebra of $G$ and $\widetilde{G}$, and denote by $\exp:\mathfrak{g} \rightarrow G$ and $\widetilde{\exp}:\mathfrak{g} \rightarrow \widetilde{G}$ the corresponding exponential maps. These exponential maps have as their image a neighbourhood of the identity. The group $\widetilde{G}$ has a polar decomposition (see Section \ref{subsec:kakdecomposition}) $\widetilde{G}=\widetilde{K}\overline{\widetilde{A}^{+}}\widetilde{K}$ that is strongly related to the polar decomposition $G=KAK$ of $G$. It is known that the exponential map of a connected simple Lie group is a bijection from the $\mathfrak{a}$-summand of the $KAK$-decomposition on the Lie algebra level to $A$. Therefore, it follows that $\widetilde{A} \cong A$. This implies that the ``infinite covering'' part of $G$ is intrinsic to the $K$-part of the polar decomposition. It is known that $\exp:\mathfrak{k} \rightarrow K$ is surjective, because $K$ is connected and compact. Also, since $\mathfrak{k} = \mathfrak{su}(2) \oplus \mathbb{R}$ (see \cite[Lemma 9]{rawnsley}), it follows that $\widetilde{\exp}:\mathfrak{k} \rightarrow \widetilde{K}$ is surjective. We summarize these facts (based on \cite[Section IX.1]{helgasonlie}) in the following proposition.
\begin{prp} \label{prp:kakdecomposition}
  We have $G=KAK$ and $\widetilde{G}=\widetilde{K}\widetilde{A}\widetilde{K}$, where $K$ and $A$ are as above, and
\begin{equation} \nonumber
\begin{split}
  K=\exp(\mathfrak{k}), \qquad &A=\exp(\mathfrak{a}),\\
  \widetilde{K}=\widetilde{\exp}(\mathfrak{k}), \qquad & \widetilde{A}=\widetilde{\exp}(\mathfrak{a}).
\end{split}
\end{equation}
Here, $\mathfrak{k}$ and $\mathfrak{a}$ denote the Lie algebras of $K$ and $A$, respectively. The group $\widetilde{A}$ is isomorphic to $A$. We can restrict to the positive Weyl chamber, and get
\[
  \overline{\widetilde{A}^{+}}=\widetilde{\exp}(\{\diag(e^{\beta},e^{\gamma},e^{-\beta},e^{-\gamma}) \mid \beta \geq \gamma \geq 0\}),
\]
which yields the decomposition $\widetilde{G}=\widetilde{K}\overline{\widetilde{A}^{+}}\widetilde{K}$.
\end{prp}
Note that the group $\mathrm{SU}(2)$ is a natural subgroup of $\mathrm{U}(2)$. Denote by $H$ the corresponding subgroup of $K$. We also get a corresponding group $\widetilde{H}$, which is isomorphic to $H$, since $\SU(2)$ is simply connected. 
\begin{dfn}
    We define $\mathcal{C}$ to be the following class of functions:
\[
  \mathcal{C}:=\{\varphi \in C(\widetilde{G}) \mid \varphi \textrm{ is }\widetilde{H}\textrm{-bi-invariant and }\mathrm{Int}(\widetilde{K})\textrm{-invariant}\}.
\]
\end{dfn}
We refer to Section \ref{subsec:gelfandpairs} for the notions of $\widetilde{H}$-bi-invariant and $\mathrm{Int}(\widetilde{K})$-invariant functions. In the notation used in that section, we have $\mathcal{C}=C(\widetilde{H} \backslash \widetilde{G} / \widetilde{H}) \cap C(\widetilde{G}//\widetilde{K})$.

Consider the generator $\begin{pmatrix} i & 0 \\ 0 & i \end{pmatrix}$ of the Lie algebra of the center of $\U(2)$. Let $Z$ denote the corresponding element of $\mathfrak{k}$. The elements $v_t=\exp(tZ)$ and $\widetilde{v}_t=\widetilde{\exp}(tZ)$ for $t \in \mathbb{R}$ are elements of the centers of $K$ and $\widetilde{K}$, respectively. Also, the family $v_t$ is periodic with period $2\pi$. Explicitly, we have
\[
  v_t=\begin{pmatrix} \cos t & 0 & -\sin t & 0 \\ 0 & \cos t & 0 & -\sin t \\ \sin t & 0 & \cos t & 0 \\ 0 & \sin t & 0 & \cos t \end{pmatrix}.
\]
\begin{rmk} \label{rmk:alternativecharacterizationc}
  Every $k \in K$ can be written as the product $k=v_th$ for some $t \in \mathbb{R}$ and $h \in H$, and, similarly, every $\widetilde{k} \in \widetilde{K}$ can be written as the product $\widetilde{k}=\widetilde{v}_t\widetilde{h}$ for some $t \in \mathbb{R}$ and $\widetilde{h} \in \widetilde{H}$. Hence, the class $\mathcal{C}$ can also be defined in the following way:
\[
  \mathcal{C}:=\{\varphi \in C(\widetilde{G}) \mid \varphi \textrm{ is }\widetilde{H}\textrm{-bi-invariant and }\varphi(\widetilde{v}_tg\widetilde{v}_t^{-1})=\varphi(g)\,\forall g \in \widetilde{G}\,\forall t \in \mathbb{R}\}.
\]
\end{rmk}
For $\beta \geq \gamma \geq 0$, let $D(\beta,\gamma)=\diag(e^{\beta},e^{\gamma},e^{-\beta},e^{-\gamma}) \in G$, which is, as pointed out before, an element of $\overline{{A}^{+}}$. Since $\widetilde{A} \cong A$, there is one and only one element $\widetilde{D}(\beta,\gamma)$ in $\overline{\widetilde{A}^{+}}$ that surjects onto $D(\beta,\gamma) \in G$. We now show that functions in $\mathcal{C}$ are completely determined by their values at elements of the form $\widetilde{v}_t\widetilde{D}(\beta,\gamma)$. Firstly, let us prove the following lemma.
\begin{lem} \label{lem:exptzdbetagamma}
  In the realization of \eqref{eq:uncov}, we have $\widetilde{v}_t\widetilde{D}(\beta,\gamma)=(v_tD(\beta,\gamma),2t)$ for $\beta \geq \gamma \geq 0$ and $t \in \mathbb{R}$.
\end{lem}
\begin{proof}
  By the description of $\widetilde{G}$ and the fact that the covering map is a homomorphism, it follows that $\widetilde{v}_t\widetilde{D}(\beta,\gamma)=(v_tD(\beta,\gamma),s)$ for some $s \in \mathbb{R}$. Using that $(v_t^T)^{-1}=v_t$ and that $\iota$ is an algebra homomorphism, it follows that $\iota(C_{v_tD(\beta,\gamma)})= \iota(\frac{1}{2}v_t(D(\beta,\gamma)+D(-\beta,-\gamma)))=\iota(v_t)\diag(\cosh(\beta),\cosh(\gamma))$. Hence,
\[
  c(\widetilde{v}_t\widetilde{D}(\beta,\gamma))=\frac{\det(\iota(v_t))}{|\det(\iota(v_t))|},
\]
because $\det(\diag(\cosh(\beta),\cosh(\gamma)))=|\det(\diag(\cosh(\beta),\cosh(\gamma)))|$ and the determinant is multiplicative. Using the fact that $\{\widetilde{v}_\sigma\widetilde{D}(\beta,\gamma) \mid \sigma \in \mathbb{R}\}$ defines a continuous path in $\widetilde{G}$, the value of $s$ is computed by
\[
  s = \tan^{-1} \left(\frac{2\sin t \cos t}{\cos^2t-\sin^2t}\right) + 2k\pi=2t + 2k\pi
\]
for some $k \in \mathbb{Z}$. Since we can connect every element $\widetilde{v}_\sigma\widetilde{D}(\beta,\gamma)$ continuously to $\widetilde{v}_0=\widetilde{1}=(1,0)$ (by varying $\sigma$, $\beta$ and $\gamma$), it follows that $k=0$. Hence, $s=2t$.
\end{proof}
\begin{lem}
  A function in $\mathcal{C}$ is determined by its values at the elements of the form $\widetilde{v}_t\widetilde{D}(\beta,\gamma)$.
\end{lem}
\begin{proof}
Let $\varphi \in \mathcal{C}$, and let $g \in \widetilde{G}$. By the polar decomposition of $\widetilde{G}$, we can write $g=\widetilde{k}_1\widetilde{D}(\beta,\gamma)\widetilde{k}_2$ for some $\beta \geq \gamma \geq 0$ and $\widetilde{k}_1,\widetilde{k}_2 \in \widetilde{K}$. For $i=1,2$, let $t_i \in \mathbb{R}$ and $\widetilde{h}_i \in \widetilde{H}$ be so that $\widetilde{k}_i=\widetilde{v}_{t_i}\widetilde{h}_i=\widetilde{h}_i\widetilde{v}_{t_i}$. Using both invariance properties of functions in $\mathcal{C}$, we obtain
\begin{equation} \nonumber
  \varphi(g) = \varphi(\widetilde{h}_1\widetilde{v}_{t_1}\widetilde{D}(\beta,\gamma)\widetilde{v}_{t_2}\widetilde{h}_2) =\varphi(\widetilde{v}_{t_1+t_2}\widetilde{D}(\beta,\gamma)).
\end{equation}
\end{proof}
\begin{ntn} \label{ntn:notation}
The value of $\varphi \in \mathcal{C}$ at $g=(g_0,t) \in \widetilde{G}$ does not change if we multiply $g$ from the left or the right with an element of $\widetilde{H}$ or if we conjugate $g$ with an element of $\widetilde{K}$. This induces an equivalence relation on $\widetilde{G}$. Let $S_{\beta,\gamma,t}$ denote the corresponding equivalence class of the element $\widetilde{v}_{\frac{t}{2}}\widetilde{D}(\beta,\gamma)$ (note that the $t$-parameter of the equivalence class corresponds to the $t$-parameter coming from the equation $c(g_0)=e^{it}$). Also, for $\varphi \in \mathcal{C}$, we put $\dot{\varphi}(\beta,\gamma,t)=\varphi(\widetilde{v}_{\frac{t}{2}}\widetilde{D}(\beta,\gamma))$.
\end{ntn}
\begin{lem} \label{lem:cinvariance}
  The class $\mathcal{C}$ is invariant under the action of the one-parameter family $\widetilde{v}_t$. More precisely, if $\varphi \in \mathcal{C}$ and $t \in \mathbb{R}$, then $\varphi_t:\widetilde{G} \rightarrow \bbC$ defined by $\varphi_t(g)=\varphi(\widetilde{v}_tg)$ is also in $\mathcal{C}$. Clearly, for an element $\varphi \in M_0A(\widetilde{G}) \cap \mathcal{C}$, it follows that for all $t \in \mathbb{R}$, we have $\|\varphi_t\|_{M_0A(\widetilde{G})}=\|\varphi\|_{M_0A(\widetilde{G})}$.
\end{lem}
\begin{proof}
  Let $\varphi \in \mathcal{C}$. We have $\varphi_t(\widetilde{h}_1g\widetilde{h}_2) = \varphi(\widetilde{v}_t\widetilde{h}_1g\widetilde{h}_2) = \varphi(\widetilde{h}_1\widetilde{v}_tg\widetilde{h}_2) = \varphi(\widetilde{v}_tg) = \varphi_t(g)$ for all $g \in \widetilde{G}$, $t \in \mathbb{R}$ and $\widetilde{h}_1,\widetilde{h}_2 \in \mathcal{H}$. Moreover, we have $\varphi_t(\widetilde{v}_sg\widetilde{v}_s^{-1}) = \varphi(\widetilde{v}_t\widetilde{v}_sg\widetilde{v}_s^{-1}) = \varphi(\widetilde{v}_s\widetilde{v}_tg\widetilde{v}_s^{-1}) = \varphi(\widetilde{v}_tg) = \varphi_t(g)$ for all $g \in \widetilde{G}$ and $s,t \in \mathbb{R}$. This proves the invariance properties of $\mathcal{C}$ of Remark \ref{rmk:alternativecharacterizationc} for $\varphi_t$.
\end{proof}
\begin{lem} \label{lem:restrictiontoc}
  If $\widetilde{G}$ has the AP, then the approximating net can be chosen in the set $A(\widetilde{G}) \cap \mathcal{C}$.
\end{lem}
\begin{proof}
For $f \in C(\widetilde{G})$ or $f \in L^1(\widetilde{G})$, we define
\[
  f^{\mathcal{C}}(g)=\frac{1}{\pi}\int_{\mathbb{R} / \pi\mathbb{Z}} \int_{\widetilde{H}} \int_{\widetilde{H}} f(\widetilde{h}_1\widetilde{v}_tg\widetilde{v}_t^{-1}\widetilde{h}_2) d\widetilde{h}_1d\widetilde{h}_2dt, \qquad g \in \widetilde{G},
\]
where $d\widetilde{h}_1$ and $d\widetilde{h}_2$ both denote the normalized Haar measure on $\widetilde{H}$. The function $f^{\mathcal{C}}$ clearly satisfies the invariance properties of Remark \ref{rmk:alternativecharacterizationc}.

The rest of the proof is similar to the proof of \cite[Lemma 2.5]{haagerupdelaat1}.
\end{proof}
\begin{prp} \label{prp:sp2ab}
  There exist constants $C_1,C_2 > 0$ such that for all functions $\varphi$ in $M_0A(\widetilde{G}) \cap \mathcal{C}$ and $t \in \mathbb{R}$, the limit $c_{\varphi}(t)=\lim_{s \to \infty} \dot{\varphi}(2s,s,t)$ exists, and for all $\beta \geq \gamma \geq 0$, we have
\[
  |\dot{\varphi}(\beta,\gamma,t)-c_{\varphi}(t)| \leq C_1 e^{-C_2\sqrt{\beta^2+\gamma^2}}\|\varphi\|_{M_0A(\widetilde{G})}.
\]
\end{prp}
The proof of this proposition will be postponed. Using the following lemma, we will explain how the proposition implies Theorem \ref{thm:covsp2noap}.
\begin{lem} \label{lem:cmultipliersclosed}
 The space consisting of $\varphi$ in $M_0A(\widetilde{G}) \cap \mathcal{C}$ for which $c_{\varphi}(t) \equiv 0$ is $\sigma(M_0A(\widetilde{G}),M_0A(\widetilde{G})_{*})$-closed.
\end{lem}
\begin{proof}
Let $(\varphi_{\alpha})$ be a net in $M_0A(\widetilde{G}) \cap \mathcal{C}$ converging to $\varphi \in M_0A(\widetilde{G})$. It follows that for all $f \in L^1(\widetilde{G})$, we have $\langle \varphi,f \rangle=\lim_{\alpha} \langle \varphi_{\alpha},f \rangle=\lim_{\alpha} \langle \varphi_{\alpha}^{\mathcal{C}},f \rangle=\lim_{\alpha} \langle \varphi_{\alpha},f^{\mathcal{C}} \rangle=\langle \varphi,f^{\mathcal{C}} \rangle= \langle \varphi^{\mathcal{C}},f \rangle$, i.e., the space $M_0A(\widetilde{G}) \cap \mathcal{C}$ is $\sigma(M_0A(\widetilde{G}),M_0A(\widetilde{G})_{*})$-closed, since $L^1(\widetilde{G})$ is dense in $M_0A(\widetilde{G})_{*}$.

It was proved in \cite[Lemma 2.6]{haagerupdelaat1} that whenever $(X,\mu)$ is a $\sigma$-finite measure space and $v:X \rightarrow \bbR$ is a strictly positive measurable function on $X$, then the set $S:=\{ f \in L^{\infty}(X) \mid |f(x)| \leq v(x) \textrm{ a.e.}\}$ is $\sigma(L^{\infty}(X),L^1(X))$-closed. We can apply this fact to the unit ball of the space $\{\varphi \in M_0A(\widetilde{G}) \cap \mathcal{C} \mid c_{\varphi}(t) \equiv 0\}$. Indeed, the conditions are satisfied with $v$ given by Proposition \ref{prp:sp2ab} (putting $\|\varphi\|_{M_0A(\widetilde{G})} \leq 1$).

Recall the Krein-Smulian Theorem, asserting that whenever $X$ is a Banach space and $A$ is a convex subset of the dual space $X^{*}$ such that $A \cap \{x^{*} \in X^{*} \mid \|x^{*}\| \leq r\}$ is weak-* closed for every $r>0$, then $A$ is weak-* closed \cite[Theorem V.12.1]{conway}. In the case where $A$ is a vector space, which is the case here, it suffices to check the case $r=1$, i.e., the weak-* closedness of the unit ball. It follows that the space consisting of $\varphi$ in $M_0A(\widetilde{G}) \cap \mathcal{C}$ for which $c_{\varphi}(t) \equiv 0$ is $\sigma(M_0A(\widetilde{G}),M_0A(\widetilde{G})_{*})$-closed.
\end{proof}
\begin{proof}[Proof of Theorem \ref{thm:covsp2noap} using Proposition \ref{prp:sp2ab}.]
By Lemma \ref{lem:restrictiontoc}, it follows that if there is no net in $A(\widetilde{G}) \cap \mathcal{C}$ that approximates the constant function $1$ in the $\sigma(M_0A(\widetilde{G}),M_0A(\widetilde{G})_{*})$-topology, then $\widetilde{G}$ does not have the AP. However, since the space $\{\varphi \in M_0A(\widetilde{G}) \cap \mathcal{C} \mid c_{\varphi}(t) \equiv 0\}$ is $\sigma(M_0A(\widetilde{G}),M_0A(\widetilde{G})_{*})$-closed by Lemma \ref{lem:cmultipliersclosed}, it follows immediately that the constant function $1$ cannot be approximated by such a net.
\end{proof}
The rest of this section will be devoted to proving Proposition \ref{prp:sp2ab}. Hereto, we identify certain pairs of groups in $\widetilde{G}$, as was also done for $G$ in \cite{haagerupdelaat1}. However, since $\widetilde{K}$ is not compact (unlike $K$ in $G$), one of the pairs we consider here is slightly different.

Firstly, note that $\U(1)$ is contained as a subgroup in $\SU(2)$ by the embedding
\begin{equation} \label{eq:embeddingu1su2}
  \left( \begin{array}{cc} e^{i\nu} & 0 \\ 0 & e^{-i\nu} \end{array} \right) \hookrightarrow \SU(2),
\end{equation}
where $\nu \in \bbR$. We point out that the quotient of $\SU(2)$ with respect to the equivalence relation $g \sim kgk^{-1}$ for $k \in \U(1)$ is homeomorphic to the closed unit disc $\overline{\mathbb{D}}$ in the complex plane. This homeomorphism is given by
\begin{equation} \label{eq:homeomorphismdisc}
  z = \begin{pmatrix} z_{11} & z_{12} \\ z_{21} & z_{22} \end{pmatrix} \mapsto z_{11}.
\end{equation}
Let $H_0$ denote the corresponding subgroup of $H$. It can be proved that $(H,H_0)$ is a strong Gelfand pair (see Section \ref{subsec:gelfandpairs}). However, because the theory on strong Gelfand pairs is not as well-developed as the theory of Gelfand pairs, we use a more explicit approach, and prove the things we need in a more ad hoc manner.

For $l,m \in \mathbb{Z}_{\geq 0}$, consider the so-called disc polynomials (see \cite{koornwinder}) $h_{l,m}^0:\overline{\mathbb{D}} \rightarrow \bbC$ from the closed unit disc $\overline{\mathbb{D}}$ to $\bbC$, given by
\[
  h_{l,m}^0(z)=\left\{ \begin{array}{ll} z^{l-m} P_m^{(0,l-m)}(2|z|^2-1) & \qquad l \geq m, \\ \overline{z}^{m-l} P_l^{(0,m-l)}(2|z|^2-1) & \qquad l < m. \end{array} \right.
\]
where $P_n^{(\alpha,\beta)}$ denotes the $n^{\textrm{th}}$ Jacobi polynomial.

Recall that a function $f:X \rightarrow Y$ from a metric space $X$ to a metric space $Y$ is H\"older continuous with exponent $\alpha>0$ if there exists a constant $C>0$ such that $d_Y(f(x_1),f(x_2))\leq Cd_X(x_1,x_2)^{\alpha}$, for all $x_1,x_2 \in X$. The following result (see \cite[Corollary 3.5]{haagerupdelaat1}) gives H\"older continuity with exponent $\frac{1}{4}$ of the functions $h_{l,m}^0$ on the circle in $\mathbb{D}$ centered at the origin with radius $\frac{1}{\sqrt{2}}$, with a constant independent of $l$ and $m$. It is a corollary of results of de first named author and Schlichtkrull \cite{hsjacobi}.
\begin{lem} \label{lem:hoelderhpq}
  For all $l,m \geq 0$, we have
\[
  \biggl\vert h_{l,m}^0\left(\frac{e^{i\theta_1}}{\sqrt{2}}\right)-h_{l,m}^0\left(\frac{e^{i\theta_2}}{\sqrt{2}}\right) \biggr\vert \leq \tilde{C}|\theta_1-\theta_2|^{\frac{1}{4}}
\]
for all $\theta_1,\theta_2 \in [0,2\pi)$, where $\tilde{C}$ is a constant independent of $l$ and $m$.
\end{lem}
We now prove the following decomposition result.
\begin{lem} \label{lem:hoeldersu2u1}
  Let $\varphi \in M_0A(\mathrm{SU}(2) \slash \slash \U(1))$ (recall the embedding \eqref{eq:embeddingu1su2}). Let
\[
  z = \left( \begin{array}{cc} z_{11} & z_{12} \\ z_{21} & z_{22} \end{array} \right) \in \SU(2).
\]
Then $\varphi(z)=\varphi^0(z_{11})$ for a certain function $\varphi^0:\mathbb{D} \rightarrow \bbC$, and
\[
  \varphi^0 = \sum_{l,m \geq 0} c_{l,m}h_{l,m}^0
\]
such that $\sum_{l,m \geq 0} |c_{l,m}|=\|\varphi\|_{M_0A(\SU(2))}$. Moreover, $\varphi^0$ satisfies
\[
  \biggl\vert \varphi^0\left(\frac{e^{i\theta_1}}{\sqrt{2}}\right)-\varphi^0\left(\frac{e^{i\theta_2}}{\sqrt{2}}\right) \biggr\vert \leq \tilde{C}|\theta_1-\theta_2|^{\frac{1}{4}}\|\varphi\|_{M_0A(\SU(2))}
\]
for all $\theta_1,\theta_2 \in [0,2\pi)$.
\end{lem}
\begin{proof}
Let $L \cong \U(1)$ denote the subgroup of $\U(2)$ given by the elements of the form
\[
  l_{\theta} = \begin{pmatrix} 1 & 0 \\ 0 & e^{i\theta} \end{pmatrix}, \qquad \theta \in \mathbb{R}.
\]
Note that $(\U(2),L)$ is the Gelfand pair that played an important role in \cite{haagerupdelaat1}. We now prove that there is an isometric isomorphism between $M_0A(\mathrm{SU}(2) // \U(1))$ and $M_0A(L \backslash \U(2) / L)$.

Let $\Phi:M_0A(L \backslash \U(2) / L) \rightarrow M_0A(\SU(2) // \U(1))$ be the map given by $\varphi \mapsto \tilde{\varphi}$, where $\tilde{\varphi}=\varphi\vert_{\SU(2)}$. It is clear that $\tilde{\varphi} \in M_0A(\SU(2) // \U(1))$ and that $\Phi$ is norm-decreasing.

Write $\U(2) = \SU(2) \rtimes L$ by the action given by multiplication from the right, i.e., $g=hl$, where $g \in \U(2)$, $h \in \SU(2)$ and $l \in L$. Consider the map $\Psi:M_0A(\SU(2) // \U(1)) \rightarrow M_0A(L \backslash \U(2) / L)$ given by $\varphi \mapsto \psi$, where $\psi(g)=\varphi(h)$ if $g=hl$ according to the unique factorization that follows from $\U(2) = \SU(2) \rtimes L$. It follows that $\psi(l_1hl_2)=\varphi(h)$ for all $h \in \SU(2)$ and $l_1,l_2 \in L$. Indeed, $\psi(l_1hl_2)=\psi(l_1hl_1^{-1}l_1l_2)=\varphi(l_1hl_1^{-1})=\varphi(h)$, since $lhl^{-1} \in \SU(2)$ for all $h \in \SU(2)$ and $l \in L$. From this, it follows that $\psi((h_2l_2)^{-1}h_1l_1)=\psi(l_2^{-1}h_2^{-1}h_1l_1)=\varphi(h_2^{-1}h_1)$. Let now $P,Q:\SU(2) \rightarrow \mathcal{H}$ be bounded continuous maps such that $\varphi(h_2^{-1}h_1)=\langle P(h_1),Q(h_2) \rangle$ for all $h_1,h_2 \in \SU(2)$ and $\|\varphi\|_{M_0A(\SU(2))}=\|P\|_{\infty}\|Q\|_{\infty}$. This is possible by the result of Bo\.zejko and Fendler mentioned in Section \ref{subsec:cbfmcgp}. It follows from this that also the map $\Psi$ is norm-decreasing, since the maps $\tilde{P}(hl)=P(h)$ and $\tilde{Q}(hl)=Q(h)$ give maps such that $\psi((h_2l_2)^{-1}h_1l_1)=\langle \tilde{P}(h_1l_1),\tilde{Q}(h_2l_2) \rangle$ for all $h_1,h_2 \in \SU(2)$ and $l_1,l_2 \in L$. Moreover, it is easy to check that $\Phi$ and $\Psi$ are each other's inverses.

From Proposition \ref{prp:cbfmcgp} we get a decomposition of elements of $M_0A(L \backslash \U(2) / L)$ in terms of the functions $h_{l,m}$, as was also explained in \cite[Section 3]{haagerupdelaat1}. Indeed $(\U(2),L)$ is a compact Gelfand pair. Applying the map $\Phi$ to this decomposition, i.e., restricting to $\SU(2)$, and by using the homeomorphism of \eqref{eq:homeomorphismdisc}, it follows that for $\varphi \in M_0A(\SU(2) // \U(1))$ we have $\varphi(h)=\varphi^0(h_{11})$ for a certain function $\varphi^0:\mathbb{D} \rightarrow \bbC$, and $\varphi^0 = \sum_{l,m \geq 0} c_{l,m}h_{l,m}^0$ such that $\sum_{l,m \geq 0} |c_{l,m}|=\|\varphi\|_{M_0A(\SU(2))}$. The last assertion of the lemma follows directly from Lemma \ref{lem:hoelderhpq}.
\end{proof}
\begin{rmk}
This lemma shows that the disc polynomials act like analogues of spherical functions for the strong Gelfand pair $(\SU(2),\U(1))$. The disc polynomials also occur as the spherical functions of the Gelfand pair $(\U(2),L)$, where $L$ is as above. It turns out that there is a general connection between the spherical functions for certain Gelfand pairs and their analogues for certain strong Gelfand pairs. A brief account on this connection is given in Appendix \ref{sec:sgp}.
\end{rmk}
Note that we can identify $M_0A(H // H_0)$ with $M_0A(\SU(2) // \U(1))$.
\begin{prp} \label{prp:psihyp}
  Let $\varphi \in M_0A(\widetilde{G}) \cap \mathcal{C}$. For $\alpha \geq 0$, let $\psi_{\alpha}:H \rightarrow \bbC$ be given by $h \mapsto \varphi(\widetilde{D}(\alpha,0)\widetilde{h}\widetilde{D}(\alpha,0))$. This function is an element of $M_0A(H // H_0)$, and $\|\psi_{\alpha}\|_{M_0A(H)} \leq \|\varphi\|_{M_0A(\widetilde{G})}$.
\end{prp}
\begin{proof}
Let $L$ be the subgroup of $\U(2)$ as in Lemma \ref{lem:hoeldersu2u1}, and let $K_0$ (resp.~$\widetilde{K}_0$) be the corresponding subgroup of $K$ (resp.~$\widetilde{K}$). For $\widetilde{h}_0 \in \widetilde{H}_0$, we can write $\widetilde{h}_0=\widetilde{k}_0\widetilde{v}_t$ for some $\widetilde{k}_0 \in \widetilde{K}_0$ and $t \in \mathbb{R}$. Since $\widetilde{k}_0$ is the exponential of an element in the Lie algebra (note that this does not hold for every element in $\widetilde{G}$), and since this element of the Lie algebra commutes with the Lie algebra element corresponding to $\widetilde{D}(\alpha,0)$, the elements also commute on the Lie group level. Hence, for all $h \in H$ and $h_0=k_0v_t \in H_0$,
\begin{equation} \nonumber
\begin{split}
  \psi_{\alpha}(h_0hh_0^{-1}) &= \varphi(\widetilde{D}(\alpha,0)\widetilde{k}_0\widetilde{v}_t\widetilde{h}\widetilde{v}_t^{-1}\widetilde{k}_0^{-1}\widetilde{D}(\alpha,0)) \\
    &=\varphi(\widetilde{k}_0\widetilde{D}(\alpha,0)\widetilde{h}\widetilde{D}(\alpha,0)\widetilde{k}_0^{-1}) \\
    &=\varphi(\widetilde{D}(\alpha,0)\widetilde{h}\widetilde{D}(\alpha,0)) \\
    &=\psi_{\alpha}(h),
\end{split}
\end{equation}
so $\psi_{\alpha}$ is an element of $C(H \slash\slash H_0)$. The statement on the norms follows in the same way as in \cite[Lemma 3.7]{haagerupdelaat1}.
\end{proof}
Suppose that $\beta \geq \gamma \geq 0$, and let $D(\beta,\gamma)$ and $\widetilde{D}(\beta,\gamma)$ be as before. Let $S_{\beta,\gamma,t}$ be as in Notation \ref{ntn:notation}. In what follows, let $\|.\|_{HS}$ denote the Hilbert-Schmidt norm of an operator, and let $h \in H$ be such that
\begin{equation} \label{eq:hform}
  \iota(h)=\begin{pmatrix} a+ib & -c+id \\ c+id & a-ib \end{pmatrix},
\end{equation}
with $a^2+b^2+c^2+d^2=1$. The following is an easy adaptation of \cite[Lemma 3.8]{haagerupdelaat1}.
\begin{lem} \label{lem:kakeqs}
  Let $g=(g_0,t) \in \widetilde{G}$. Then $g \in S_{\beta,\gamma,t}$, where $\beta,\gamma \in \bbR$ are uniquely determined by the condition $\beta \geq \gamma \geq 0$ together with the equations
  \begin{equation} \nonumber
  \begin{split}
    \sinh^2 \beta + \sinh^2 \gamma &= \frac{1}{8}\|g_0-(g_0^T)^{-1}\|_{HS}^2, \\  
    \sinh^2 \beta \sinh^2 \gamma &= \frac{1}{16}\det(g_0-(g_0^T)^{-1}).
  \end{split}
  \end{equation}
\end{lem}
\begin{lem} \label{lem:hyperbolaseqs}
  Let $\alpha > 0$ and $\beta \geq \gamma \geq 0$. If $\widetilde{h} \in \widetilde{H}$ is such that the corresponding $h \in H$ satisfies \eqref{eq:hform}, and $c=\sqrt{1-a^2-b^2}=\frac{1}{\sqrt{2}}$ and $d=0$, then $\widetilde{D}(\alpha,0)\widetilde{h}\widetilde{D}(\alpha,0) \in S_{\beta,\gamma,t}$ if and only if
\begin{equation} \label{eq:hyperbolaseqs}
  \begin{split}
    & \sinh \beta \sinh \gamma = \frac{1}{2}\sinh^2 \alpha (1-a^2-b^2), \\  
    & \sinh \beta - \sinh \gamma = \sinh(2\alpha)|a|,\\
    & t=-\tan^{-1} \left(\frac{2ab}{\coth^2(\alpha)+a^2-b^2}\right).
  \end{split}
  \end{equation}
\end{lem}
\begin{proof}
Let $\alpha > 0$ and $\beta \geq \gamma \geq 0$. By Lemma \ref{lem:kakeqs}, $\widetilde{D}(\alpha,0)\widetilde{h}\widetilde{D}(\alpha,0) \in S_{\beta,\gamma,t}$ if and only if
\begin{equation} \label{eq:hypeq1}
\begin{split}
    \sinh^2 \beta + \sinh^2 \gamma &= \frac{1}{8}\|D(\alpha,0)hD(\alpha,0)-D(\alpha,0)^{-1}hD(\alpha,0)^{-1}\|_{HS}^2 \\
      &=\sinh^2(2\alpha)a^2 + \sinh^2\alpha,
\end{split}
\end{equation}
and
\begin{equation} \label{eq:hypeq2}
\begin{split}
    \sinh^2 \beta \sinh^2 \gamma &= \frac{1}{16}\det(D(\alpha,0)hD(\alpha,0)-D(\alpha,0)^{-1}hD(\alpha,0)^{-1}) \\
      &=\frac{1}{4}\sinh^4\alpha,
\end{split}
\end{equation}
and, using the explicit expression of \eqref{eq:circlefunction},
\begin{equation} \label{eq:hypeq3}
\begin{split}
    e^{it} &= \frac{\det(\iota(C_{D(\alpha,0)hD(\alpha,0)}))}{|\det(\iota(C_{D(\alpha,0)hD(\alpha,0)}))|}.
\end{split}
\end{equation}
The fact that the first two equations of \eqref{eq:hyperbolaseqs} hold if and only if \eqref{eq:hypeq1} and \eqref{eq:hypeq2} hold was proved in \cite[Lemma 3.9]{haagerupdelaat1}. The rest of the proof consists of computing $t$. From \eqref{eq:hypeq3}, it follows that
\begin{equation} \label{eq:tk}
  t=\arg(\det(\iota(C_{D(\alpha,0)hD(\alpha,0)})))+2k\pi
\end{equation}
for some $k \in \mathbb{Z}$. It is elementary to check that
\begin{equation} \nonumber
\begin{split}
  \iota(C_{D(\alpha,0)hD(\alpha,0)}) &= \iota(\frac{1}{2}\left(D(\alpha,0)hD(\alpha,0)+D(-\alpha,0)hD(-\alpha,0)\right)) \\
    &= \begin{pmatrix} \cosh(2\alpha) a + ib & -\frac{\cosh(\alpha)}{\sqrt{2}} \\ \frac{\cosh(\alpha)}{\sqrt{2}} & a - ib \end{pmatrix}.
\end{split}
\end{equation}
Computing the determinant of this matrix yields
\[
  \det(\iota(C_{D(\alpha,0)hD(\alpha,0)}))=\cosh(2\alpha)a^2+b^2+iab-iab\cosh(2\alpha)+\frac{1}{2}\cosh^2(\alpha).
\]
Determining the argument is done by taking the inverse tangent of the the imaginary part of this determinant divided by its real part, which yields
\begin{equation} \nonumber
\begin{split}
  &\arg(\det(\iota(C_{D(\alpha,0)hD(\alpha,0)}))) = \tan^{-1} \left(\frac{ab(1-\cosh(2\alpha))}{\cosh(2\alpha)a^2+b^2+\frac{1}{2}\cosh^2(\alpha)}\right) \\
  &= -\tan^{-1} \left(\frac{2ab\sinh^2(\alpha)}{a^2+2a^2\sinh^2(\alpha)+b^2\cosh^2(\alpha)-b^2\sinh^2(\alpha)+\frac{1}{2}\cosh^2(\alpha)}\right) \\
  &= -\tan^{-1} \left(\frac{2ab\sinh^2(\alpha)}{(a^2-b^2)\sinh^2(\alpha)+(\frac{1}{2}+a^2+b^2)\cosh^2(\alpha)}\right) \\
  &= -\tan^{-1} \left(\frac{2ab}{\coth^2(\alpha) + a^2-b^2}\right).
\end{split}
\end{equation}
Since $\coth^2(\alpha) \geq 1$ for all $\alpha > 0$, the argument of the inverse tangent is clearly a bounded function. Hence, the value of $k$ in \eqref{eq:tk} is the same for the whole family of elements of the form $\widetilde{D}(\alpha,0)\widetilde{h}\widetilde{D}(\alpha,0)$. Since there exists a continuous path from any $\widetilde{D}(\alpha,0)\widetilde{h}\widetilde{D}(\alpha,0)$ to the identity element of $\widetilde{G}$, it follows that $k=0$.
\end{proof}

We now consider a different pair of groups in $\widetilde{G}$. The natural embedding of $\SO(2)$ in $\SU(2)$ gives rise to a subgroup $H_1$ of $H$ and to a subgroup $\widetilde{H}_1$ of $\widetilde{H}$. The pair $(H,H_1)$ is a compact Gelfand pair and was used in \cite{haagerupdelaat1} as well. If $h \in \SU(2)$ satisfies \eqref{eq:hform}, then the double cosets of $\SO(2)$ in $\SU(2)$ are labeled by $a^2-b^2+c^2-d^2$. Hence, every $\SO(2)$-bi-invariant function $\chi:\SU(2) \rightarrow \mathbb{C}$ is of the form $\chi(h)=\chi^0(a^2-b^2+c^2-d^2)$ for a certain function $\chi^0:[-1,1] \rightarrow \mathbb{C}$, since $\SO(2) \backslash \SU(2) \slash \SO(2) \cong [-1,1]$. The spherical functions for this Gelfand pair are indexed by $n \geq 0$, and given by $P_n(a^2-b^2+c^2-d^2)$, where $P_n$ denotes the $n^{\textrm{th}}$ Legendre polynomial. For details, we refer to \cite{haagerupdelaat1}.

The following estimate was proved (in this explicit form) in \cite[Lemma 3.11]{haagerupdelaat1}. Similar estimates were already proved in \cite{lafforguestrengthenedpropertyt}, and, as was remarked in \cite[Remark 3.12]{haagerupdelaat1}, they can also be obtained from Szeg\"o's book \cite{szegoe}.
\begin{lem} \label{lem:legendreestimates}
	For all non-negative integers $n$,
	\[
	  |P_n(x)-P_n(y)|\leq 4|x-y|^{\frac{1}{2}}
	\]
for $x,y \in [-\frac{1}{2},\frac{1}{2}]$, i.e., the Legendre polynomials are H\"older continuous on $[-\frac{1}{2},\frac{1}{2}]$ with exponent $\frac{1}{2}$.
\end{lem}
\begin{lem}
  Let $\varphi \in M_0A(\SO(2) \backslash \SU(2) / \SO(2))$. Suppose that $h \in \SU(2)$ is of the form
\[
  h=\left( \begin{matrix} a+ib & -c+id \\ c+id & a-ib \end{matrix} \right),
\]
where $a,b,c,d \in \mathbb{R}$ are such that $a^2+b^2+c^2+d^2=1$. Then $\varphi(h)=\varphi^0(r)$, where $r=a^2-b^2+c^2-d^2$, for a certain function $\varphi^0:[-1,1] \rightarrow \bbC$, and
\[
  \varphi^0 = \sum_{n \geq 0} c_nP_n
\]
such that $\sum_{n \geq 0} |c_n|=\|\varphi\|_{M_0A(\SU(2))}$. Moreover, $\varphi^0$ satisfies
\[
  |\varphi^0(r_1)-\varphi^0(r_2)| \leq 4|r_1-r_2|^{\frac{1}{2}}\|\varphi\|_{M_0A(\widetilde{G})}
\]
for all $r_1,r_2 \in [-\frac{1}{2},\frac{1}{2}]$.
\end{lem}
The above lemma follows directly from Proposition \ref{prp:cbfmcgp} and Lemma \ref{lem:legendreestimates}. Note that we can identify $M_0A(H_1 \backslash H / H_1)$ with $M_0A(\SO(2) \backslash \SU(2) \slash \SO(2))$.
\begin{ntn}
  In what follows, we use the notation $v=v_{\frac{\pi}{4}}$ and $\widetilde{v}=\widetilde{v}_{\frac{\pi}{4}}$.
\end{ntn}
The proof of the following proposition is similar to the proof of Proposition \ref{prp:psihyp}.
\begin{prp} \label{prp:chialpha}
  Let $\varphi \in M_0A(\widetilde{G}) \cap \mathcal{C}$. For $\alpha \geq 0$, let $\chi_{\alpha}^{\prime}:H \rightarrow \mathbb{C}$ be given by $h \mapsto \varphi(\widetilde{D}(\alpha,\alpha)\widetilde{v}\widetilde{h}\widetilde{D}(\alpha,\alpha))$, and let $\chi_{\alpha}^{\prime\prime}:H \rightarrow \mathbb{C}$ be given by $h \mapsto \varphi(\widetilde{D}(\alpha,\alpha)\widetilde{v}^{-1}\widetilde{h}\widetilde{D}(\alpha,\alpha))$. These functions are elements of $M_0A(H_1 \backslash H / H_1)$, and $\|\chi_{\alpha}^{\prime}\|_{M_0A(H)} \leq \|\varphi\|_{M_0A(\widetilde{G})}$ and $\|\chi_{\alpha}^{\prime\prime}\|_{M_0A(H)} \leq \|\varphi\|_{M_0A(\widetilde{G})}$. 
\end{prp}
Suppose that $\beta \geq \gamma \geq 0$, and let $D(\beta,\gamma)$ and $\widetilde{D}(\beta,\gamma)$ be as before.
\begin{lem} \label{lem:circleseqs}
Let $\alpha > 0$ and $\beta \geq \gamma \geq 0$. If $\widetilde{h}$ is such that the corresponding $h$ satisfies \eqref{eq:hform}, then $\widetilde{D}(\alpha,\alpha)\widetilde{v}\widetilde{h}\widetilde{D}(\alpha,\alpha) \in S_{\beta,\gamma,t}$ if and only if
\begin{equation} \nonumber
  \begin{cases}
    & \sinh^2 \beta + \sinh^2 \gamma = \sinh^2 (2\alpha), \\  
    & \sinh \beta \sinh \gamma = \frac{1}{2}\sinh^2(2\alpha)|r|, \\
    & t=\frac{\pi}{2}-\tan^{-1} \left(\frac{\sinh^2(2\alpha)}{2\cosh(2\alpha)}r \right),
  \end{cases}
  \end{equation}
and 
$\widetilde{D}(\alpha,\alpha)\widetilde{v}^{-1}\widetilde{h}\widetilde{D}(\alpha,\alpha) \in S_{\beta,\gamma,t}$ if and only if
\begin{equation} \label{eq:lemeq1}
  \begin{cases}
    & \sinh^2 \beta + \sinh^2 \gamma = \sinh^2 (2\alpha), \\  
    & \sinh \beta \sinh \gamma = \frac{1}{2}\sinh^2(2\alpha)|r|, \\
    & t=-\frac{\pi}{2}+\tan^{-1} \left(\frac{\sinh^2(2\alpha)}{2\cosh(2\alpha)}r \right),
  \end{cases}
  \end{equation}
where $r=a^2-b^2+c^2-d^2$.
\end{lem}
\begin{proof}
Let $\alpha > 0$ and $\beta \geq \gamma \geq 0$. By Lemma \ref{lem:kakeqs}, $\widetilde{D}(\alpha,\alpha)\widetilde{v}\widetilde{h}\widetilde{D}(\alpha,\alpha) \in S_{\beta,\gamma,t}$ if and only if
\begin{equation} \label{eq:cireq1}
\begin{split}
    \sinh^2 \beta + \sinh^2 \gamma &= \frac{1}{8}\|D(\alpha,\alpha)vhD(\alpha,\alpha)-D(\alpha,\alpha)^{-1}vhD(\alpha,\alpha)^{-1}\|_{HS}^2 \\
      &=\sinh^2(2\alpha),
\end{split}
\end{equation}
and
\begin{equation} \label{eq:cireq2}
\begin{split}
    \sinh^2 \beta \sinh^2 \gamma &= \frac{1}{16}\det(D(\alpha,\alpha)vhD(\alpha,0)-D(\alpha,\alpha)^{-1}vhD(\alpha,\alpha)^{-1}) \\
      &=\frac{1}{4}\sinh^4(2\alpha)r^2,
\end{split}
\end{equation}
and, using the explicit expression of \eqref{eq:circlefunction},
\begin{equation} \label{eq:cireq3}
\begin{split}
    e^{it} &= \frac{\det(\iota(C_{D(\alpha,\alpha)vhD(\alpha,\alpha)}))}{|\det(\iota(C_{D(\alpha,\alpha)vhD(\alpha,\alpha)}))|}.
\end{split}
\end{equation}
The first two equations of \eqref{eq:lemeq1} are now obvious. The last part of the proof consists of computing $t$. From \eqref{eq:cireq3}, it follows that
\[
t=\arg(\det(\iota(C_{D(\alpha,\alpha)vhD(\alpha,\alpha)})))+2k\pi
\]
for some $k \in \mathbb{Z}$. It is elementary to check that
\begin{equation} \nonumber
\begin{split}
  \iota(C_{D(\alpha,\alpha)vhD(\alpha,\alpha)}) = \iota(\frac{1}{2}(D(\alpha,\alpha)vhD(\alpha,\alpha)+D(-\alpha,-\alpha)vhD(-\alpha,-\alpha))) \\
    = \frac{1}{\sqrt{2}} \begin{pmatrix} \cosh(2\alpha) (a-b) + i(a+b) & -\cosh(2\alpha)(c+d)-i(c-d) \\ \cosh(2\alpha)(c-d) + i(c+d) & \cosh(2\alpha)(a+b) + i(a-b) \end{pmatrix}.
\end{split}
\end{equation}
Computing the determinant of this matrix yields
\begin{equation} \nonumber
\begin{split} 
  &\det(\iota(C_{D(\alpha,\alpha)vhD(\alpha,\alpha)})) \\
  &= \frac{1}{2}(a^2-b^2+c^2-d^2)(\cosh^2(2\alpha)-1)+i(a^2+b^2+c^2+d^2)\cosh(2\alpha).
\end{split}
\end{equation}
Determining the argument is done by taking the inverse tangent of the the imaginary part of this determinant divided by its real part. By $\arg(x+iy)=\tan^{-1}(\frac{y}{x})=\frac{\pi}{2}-\tan^{-1}(\frac{x}{y})$ for $x \neq 0$ and $y > 0$, we obtain
\begin{equation} \nonumber
\begin{split}
  \arg(\det(\iota(C_{D(\alpha,\alpha)vhD(\alpha,\alpha)}))) &= \frac{\pi}{2} - \tan^{-1} \left(\frac{(a^2-b^2+c^2-d^2)(2\cosh^2(2\alpha)-1)}{(a^2+b^2+c^2+d^2)\cosh(2\alpha)}\right) \\
    &= \frac{\pi}{2} - \tan^{-1} \left(\frac{\sinh^2(2\alpha)}{2\cosh(2\alpha)}r \right).
\end{split}
\end{equation}
The second inclusion we have to consider, i.e., $\widetilde{D}(\alpha,\alpha)\widetilde{v}^{-1}\widetilde{h}\widetilde{D}(\alpha,\alpha) \in S_{\beta,\gamma,t}$ is very similar. It is easy to check that this holds if and only if \eqref{eq:cireq1} and \eqref{eq:cireq2} hold. As for the value of $t$, it is very similar to the first case. Indeed, it is again elementary to check that
\begin{equation} \nonumber
\begin{split}
  \iota(C_{D(\alpha,\alpha)v^{-1}hD(\alpha,\alpha)}) = \iota(\frac{1}{2}(D(\alpha,\alpha)v^{-1}hD(\alpha,\alpha)+D(-\alpha,-\alpha)v^{-1}hD(-\alpha,-\alpha))) \\
    = \frac{1}{\sqrt{2}} \begin{pmatrix} \cosh(2\alpha) (a+b) - i(a-b) & -\cosh(2\alpha)(c-d)+i(c+d) \\ \cosh(2\alpha)(c+d) - i(c-d) & \cosh(2\alpha)(a-b) - i(a+b) \end{pmatrix}.
\end{split}
\end{equation}
It follows that
\[
  \arg(\det(\iota(C_{D(\alpha,\alpha)v^{-1}hD(\alpha,\alpha)}))) = -\frac{\pi}{2} + \tan^{-1} \left(\frac{\sinh^2(2\alpha)}{2\cosh(2\alpha)}r \right).
\]
By an argument similar to the one in the proof of Lemma \ref{lem:hyperbolaseqs}, it follows that $k=0$, giving the correct values of $t$.
\end{proof}
We will now prove that multipliers in $M_0A(\widetilde{G}) \cap \mathcal{C}$ are almost constant on certain paths in the groups.
\begin{prp} \label{prp:tdependence}
  Let $\varphi \in M_0A(\widetilde{G}) \cap \mathcal{C}$. If $\alpha > 0$ and $|\tau_1-\tau_2| \leq \frac{\pi}{2}$, then
\[
  |\dot{\varphi}(2\alpha,0,\tau_1)-\dot{\varphi}(2\alpha,0,\tau_2)| \leq 24 e^{-\alpha} \|\varphi\|_{M_0A(\widetilde{G})}.
\]
\end{prp}
In order to prove this result, we need the following lemma.
\begin{lem} \label{lem:tdependence}
  Let $\varphi \in M_0A(\widetilde{G}) \cap \mathcal{C}$, let $\alpha \geq 2$ and $\tau \in [-\frac{\pi}{4},\frac{\pi}{4}]$. Let $r=-\frac{2\cosh(2\alpha)}{\sinh^2(2\alpha)}\tan(\tau)$, and let $\beta \geq \gamma \geq 0$ be the unique numbers for which
\begin{equation} \nonumber
\begin{split}
  \sinh \beta &= \frac{1}{2}\sinh(2\alpha)(\sqrt{1+|r|}+\sqrt{1-|r|}),\\
  \sinh \gamma &= \frac{1}{2}\sinh(2\alpha)(\sqrt{1+|r|}-\sqrt{1-|r|}).
\end{split}
\end{equation}
Then
\[
  |\dot{\varphi}(\beta,\gamma,\tau)-\dot{\varphi}(2\alpha,0,0)| \leq 12e^{-\alpha}\|\varphi\|_{M_0A(\widetilde{G})}.
\]
\end{lem}
\begin{proof}
One easily checks that $\sinh^2 \beta + \sinh^2 \gamma=\sinh^2(2\alpha)$ and $2\sinh\beta\sinh\gamma=\sinh^2(2\alpha)|r|$. Put
\[
g(r)=\widetilde{D}(\alpha,\alpha)\widetilde{v}\widetilde{h}(r)\widetilde{D}(\alpha,\alpha) \in S_{\beta,\gamma,\tau^{\prime}},
\]
where
\[
  \tau^{\prime}=\frac{\pi}{2}-\tan^{-1} \left( \frac{\sinh^2(2\alpha)}{2\cosh(2\alpha)}r \right)=\frac{\pi}{2}+\tau
\]
and $\widetilde{h}(r)$ is any element in $\widetilde{H}$ satisfying $a^2-b^2+c^2-d^2=r$. By Proposition \ref{prp:chialpha}, we obtain
\[
  |\varphi(g(r))-\varphi(g(0))| \leq 4 |r|^{\frac{1}{2}} \|\varphi\|_{M_0A(\widetilde{G})},
\]
provided that $|r| \leq \frac{1}{2}$. Since $g(0)$ corresponds to $r=0$, it follows that the corresponding $\tau^{\prime}$ equals $\frac{\pi}{2}$. Hence, $g(0) \in S_{2\alpha,0,\frac{\pi}{2}}$.

Hence, by the invariance property of $\mathcal{C}$ of Lemma \ref{lem:cinvariance},
\[
  |\dot{\varphi}(\beta,\gamma,\tau)-\dot{\varphi}(2\alpha,0,0)| \leq 4|r|^{\frac{1}{2}}\|\varphi\|_{M_0A(\widetilde{G})},
\]
provided that $|r| \leq \frac{1}{2}$. However, since $|\tau| \leq \frac{\pi}{4}$, we have $|\tan{\tau}| \leq 1$. It follows that $|r| \leq \frac{2\cosh(2\alpha)}{\sinh^2(2\alpha)} \leq \frac{4e^{2\alpha}(1+e^{-4\alpha})}{e^{4\alpha}(1-e^{-4\alpha})^2} \leq 4e^{-2\alpha}\left(\frac{1+e^{-8}}{(1-e^{-8})^2}\right) \leq 5e^{-2\alpha}$ for $\alpha \geq 2$. Then $|r| \leq 5e^{-4} < \frac{1}{2}$. This implies that
\[
  |\dot{\varphi}(\beta,\gamma,\tau)-\dot{\varphi}(2\alpha,0,0)| \leq 12e^{-\alpha}\|\varphi\|_{M_0A(\widetilde{G})}.
\]
\end{proof}
\begin{proof}[Proof of Proposition \ref{prp:tdependence}.]
  Put $\tau=\frac{\tau_1-\tau_2}{2}$. It is sufficient to prove that
\[
  |\dot{\varphi}(2\alpha,0,\tau)-\dot{\varphi}(2\alpha,0,-\tau)| \leq 24e^{-\alpha}\|\varphi\|_{M_0A(\widetilde{G})}.
\]
Construct $\beta \geq \gamma \geq 0$ as in Lemma \ref{lem:tdependence}. Observe that this gives the same for $\tau$ and $-\tau$. Replacing $g(r)=\widetilde{D}(\alpha,\alpha)\widetilde{v}\widetilde{h}(r)\widetilde{D}(\alpha,\alpha)$ in that lemma by $g(r)=\widetilde{D}(\alpha,\alpha)\widetilde{v}^{-1}\widetilde{h}(r)\widetilde{D}(\alpha,\alpha)$, we obtain
\[
  |\dot{\varphi}(\beta,\gamma,-\tau)-\dot{\varphi}(2\alpha,0,0)| \leq 12e^{-\alpha}\|\varphi\|_{M_0A(\widetilde{G})}
\]
for $\alpha \geq 2$. Combining the results, we obtain
\[
  |\dot{\varphi}(\beta,\gamma,\pm\tau)-\dot{\varphi}(2\alpha,0,0)| \leq 12e^{-\alpha}\|\varphi\|_{M_0A(\widetilde{G})}.
\]
Then the invariance property of $\mathcal{C}$ (see Lemma \ref{lem:cinvariance}) implies that
\[
  |\dot{\varphi}(\beta,\gamma,0)-\dot{\varphi}(2\alpha,0,\mp\tau)| \leq 12e^{-\alpha}\|\varphi\|_{M_0A(\widetilde{G})}
\]
for $\alpha \geq 2$. Since $2e^2 \leq 24$, it follows that the desired estimate holds for every $\alpha > 0$.
\end{proof}
\begin{lem} \label{lem:betagamma}
  Let $\beta \geq \gamma \geq 0$. Then the equations
\begin{equation} \nonumber
\begin{split}
    \sinh^2(2s_1) + \sinh^2s_1 &= \sinh^2 \beta + \sinh^2\gamma, \\
    \sinh(2s_2)\sinh s_2 &= \sinh \beta \sinh \gamma
\end{split}
\end{equation}
have unique solutions $s_1=s_1(\beta,\gamma)$, $s_2=s_2(\beta,\gamma)$ in the interval $[0,\infty)$. Moreover,
\begin{equation} \label{eq:betagamma1}
   s_1 \geq \frac{\beta}{4}, \qquad s_2 \geq \frac{\gamma}{2}.
\end{equation}
\end{lem}
For a proof, see \cite[Lemma 3.16]{haagerupdelaat1}. Note that we have changed notation here.
\begin{lem} \label{lem:tdependences}
  There exists a constant $\tilde{B} > 0$ such that for $\alpha > 0$, $t \in \mathbb{R}$, $\tau \in [-\frac{\pi}{2},\frac{\pi}{2}]$, $s_1=s_1(2\alpha,0)$ chosen as in Lemma \ref{lem:betagamma}, and $\varphi \in M_0A(\widetilde{G}) \cap \mathcal{C}$, we have
\[
  |\dot{\varphi}(2s_1,s_1,t) - \dot{\varphi}(2s_1,s_1,t+\tau)| \leq \tilde{B}e^{-\frac{\alpha}{4}}\|\varphi\|_{M_0A(\widetilde{G})}.
\]
\end{lem}
\begin{proof}
  Let $t \in \mathbb{R}$, and let $\tau \in [-\frac{\pi}{2},\frac{\pi}{2}]$. Suppose first that $\alpha \geq 4$, and let $h_1 \in H$ be such that
\[
  \iota(h_1)=\frac{1}{\sqrt{2}}\left(\begin{array}{cc} 1+i & 0 \\ 0 & 1-i \end{array}\right) \in \SU(2),
\]
i.e., in the parametrization of \eqref{eq:hform}, we have $a=b=\frac{1}{\sqrt{2}}$, $c=d=0$, and, hence, $r_1=0$. By Lemma \ref{lem:circleseqs}, we have $\widetilde{D}(\alpha,\alpha)\widetilde{v}\widetilde{h}_1\widetilde{D}(\alpha,\alpha) \in S_{2\alpha,0,t^{\prime\prime}}$ for some $t^{\prime\prime} \in \mathbb{R}$. Let $s_1=s_1(2\alpha,0)$ be as in Lemma \ref{lem:betagamma}. Then $s_1 \geq 0$ and $\sinh^2(2s_1)+\sinh^2s_1=\sinh^2(2\alpha)$. Put
\[
  r_2=\frac{2\sinh (2s_1) \sinh s_1}{\sinh^2(2s_1)+\sinh^2s_1} \in [0,1],
\]
and let $h_2 \in H$ be such that
\[
  \iota(h_2)=\left(\begin{array}{cc} a_2+ib_2 & 0 \\ 0 & a_2-ib_2 \end{array}\right) \in \SU(2),
\]
where $a_2=\left(\frac{1+r_2}{2}\right)^{\frac{1}{2}}$ and $b_2=\left(\frac{1-r_2}{2}\right)^{\frac{1}{2}}$. Since $a_2^2-b_2^2=r_2$, it follows again by Lemma \ref{lem:circleseqs} that $\widetilde{D}(\alpha,\alpha)\widetilde{v}\widetilde{h}_2\widetilde{D}(\alpha,\alpha) \in S_{2s_1,s_1,t^{\prime}}$ for some $t^{\prime} \in \mathbb{R}$.

Let $\varphi \in M_0A(\widetilde{G}) \cap \mathcal{C}$, and let $\chi_{\alpha}^{\prime}(h)=\varphi(\widetilde{D}(\alpha,\alpha)\widetilde{v}\widetilde{h}\widetilde{D}(\alpha,\alpha))$ for $h \in H$ as in Proposition \ref{prp:chialpha}. By the same proposition, given the fact that $r_1=0$ and provided that $r_2 \leq \frac{1}{2}$, it follows that
\begin{equation} \label{eq:chialpha0estimate}
\begin{split}
  |\dot{\varphi}(2s_1,s_1,t^{\prime})-\dot{\varphi}(2\alpha,0,t^{\prime\prime})| & \leq |\chi_{\alpha}^{\prime}(h_1)-\chi_{\alpha}^{\prime}(h_2)| \\
    &=|\chi_{\alpha}^{\prime,0}(r_1)-\chi_{\alpha}^{\prime,0}(r_2)| \\
    & \leq 4r_2^{\frac{1}{2}}\|\varphi\|_{M_0A(\widetilde{G})},
\end{split}
\end{equation}
where $\chi_{\alpha}^{\prime,0}$ is the function on $[-1,1]$ induced by $\chi_{\alpha}^{\prime}$. Note that $r_2 \leq 2\frac{\sinh s_1}{\sinh{2s_1}}=\frac{1}{\cosh s_1}\leq 2e^{-s_1}$. By Lemma \ref{lem:betagamma}, equation \eqref{eq:betagamma1}, we obtain that $r_2 \leq 2e^{-\frac{\alpha}{2}}\leq 2e^{-2} \leq \frac{1}{2}$. In particular, \eqref{eq:chialpha0estimate} holds, and we have $r_2 \leq 2e^{-\frac{\alpha}{2}}$. The estimate above is independent on the choice of $\varphi \in M_0A(\widetilde{G}) \cap \mathcal{C}$, so by the invariance property of Lemma \ref{lem:cinvariance}, it follows that
\begin{equation} \nonumber
\begin{split}
  & |\dot{\varphi}(2s_1,s_1,t^{\prime})-\dot{\varphi}(2s_1,s_1,t^{\prime}+\tau)| \\
  & \leq |\dot{\varphi}(2s_1,s_1,t^{\prime})-\dot{\varphi}(2\alpha,0,t^{\prime\prime})| + |\dot{\varphi}(2\alpha,0,t^{\prime\prime})-\dot{\varphi}(2\alpha,0,t^{\prime\prime}+\tau)| \\
  & \qquad + |\dot{\varphi}(2\alpha,0,t^{\prime\prime}+\tau)-\dot{\varphi}(2s_1,s_1,t^{\prime}+\tau)| \\
  & \leq (8\sqrt{2}e^{-\frac{\alpha}{4}} + 24e^{-\alpha})\|\varphi\|_{M_0A(\widetilde{G})}.
\end{split}
\end{equation}
By the invariance property of Lemma \ref{lem:cinvariance}, the desired estimate follows with $\tilde{B}=8\sqrt{2}+24$.
\end{proof}
By the following two lemmas, we can estimate the difference between $\dot{\varphi}(\beta,\gamma,t)$ and the value of $\varphi$ at a certain point on the line $\{(2s,s,t) \mid s \in \mathbb{R}_{+}\}$. The method is similar to the one used in \cite[Lemma 3.17 and Lemma 3.18]{haagerupdelaat1}, but because of the $t$-dependence, there is an extra parameter. Lemma \ref{lem:tdependences} provides us with the tools to deal with this extra parameter.
\begin{lem} \label{lem:betagammacir}
  There exists a constant $B_1 > 0$ such that whenever $\beta \geq \gamma \geq 0$, $t \in \mathbb{R}$, and $s_1=s_1(\beta,\gamma)$ is chosen as in Lemma \ref{lem:betagamma}, then for all $\varphi \in M_0A(\widetilde{G}) \cap \mathcal{C}$,
\begin{equation} \nonumber
  |\dot{\varphi}(\beta,\gamma,t)-\dot{\varphi}(2s_1,s_1,t)| \leq B_1 e^{-\frac{\beta-\gamma}{8}} \|\varphi\|_{M_0A(\widetilde{G})}.
\end{equation}
\end{lem}
\begin{proof}
Let $\beta \geq \gamma \geq 0$ and $t \in \mathbb{R}$. Assume first that $\beta - \gamma \geq 8$. Let $\alpha \in [0,\infty)$ be the unique solution to $\sinh^2 \beta + \sinh^2 \gamma=\sinh^2(2\alpha)$, and observe that $2\alpha \geq \beta \geq 2$, so in particular $\alpha > 0$. Define
\[
  r_1=\frac{2\sinh \beta \sinh \gamma}{\sinh^2 \beta+\sinh^2 \gamma} \in [0,1],
\]
and $a_1=\left(\frac{1+r_1}{2}\right)^{\frac{1}{2}}$ and $b_1=\left(\frac{1-r_1}{2}\right)^{\frac{1}{2}}$. Furthermore, let $h_1 \in H$ be such that
\[
  \iota(h_1)=\left(\begin{array}{cc} a_1+ib_1 & 0 \\ 0 & a_1-ib_1 \end{array}\right) \in \SU(2),
\]
and let $\widetilde{v}$ be as before. We now have $2\sinh \beta \sinh \gamma=\sinh^2(2\alpha)r_1$, and $a_1^2-b_1^2=r_1$, so by Lemma \ref{lem:circleseqs}, we have $\widetilde{D}(\alpha,\alpha)\widetilde{v}\widetilde{h}_1\widetilde{D}(\alpha,\alpha) \in S_{\beta,\gamma,t^{\prime}}$ for some $t^{\prime} \in \mathbb{R}$.

Let now $s_1=s_1(\beta,\gamma)$ be as in Lemma \ref{lem:betagamma}. Then $s_1 \geq 0$ and $\sinh^2(2s_1)+\sinh^2s_1=\sinh^2\beta+\sinh^2\gamma=\sinh^2(2\alpha)$. Similar to the proof of Lemma \ref{lem:tdependences}, put
\[
  r_2=\frac{2\sinh (2s_1) \sinh s_1}{\sinh^2(2s_1)+\sinh^2s_1} \in [0,1]
\]
and let $h_2 \in H$ be such that
\[
  \iota(h_2)=\left(\begin{array}{cc} a_2+ib_2 & 0 \\ 0 & a_2-ib_2 \end{array}\right) \in \SU(2),
\]
where $a_2=\left(\frac{1+r_2}{2}\right)^{\frac{1}{2}}$ and $b_2=\left(\frac{1-r_2}{2}\right)^{\frac{1}{2}}$. Since $a_2^2-b_2^2=r_2$, it follows again by Lemma \ref{lem:circleseqs} that $\widetilde{D}(\alpha,\alpha)\widetilde{v}\widetilde{h}_2\widetilde{D}(\alpha,\alpha) \in S_{2s_2,s_2,t^{\prime\prime}}$ for some $t^{\prime\prime} \in \mathbb{R}$.

Now, let $\chi_{\alpha}^{\prime}(h)=\varphi(\widetilde{D}(\alpha,\alpha)\widetilde{v}\widetilde{h}\widetilde{D}(\alpha,\alpha))$ for $h \in H$ as in Proposition \ref{prp:chialpha}. By the same proposition, it follows that
\begin{equation} \label{eq:chialphaestimate}
  |\chi_{\alpha}^{\prime}(h_1)-\chi_{\alpha}^{\prime}(h_2)|=|\chi_{\alpha}^{\prime,0}(r_1)-\chi_{\alpha}^{\prime,0}(r_2)|\leq4|r_1-r_2|^{\frac{1}{2}}\|\varphi\|_{M_0A(\widetilde{G})},
\end{equation}
provided that $r_1,r_2 \leq \frac{1}{2}$. Note that $r_1 \leq \frac{2\sinh\beta\sinh\gamma}{\sinh^2\beta}=2\frac{\sinh\gamma}{\sinh\beta}$. Hence, using $\beta\geq\gamma+8\geq\gamma$, we get $r_1 \leq 2\frac{e^{\gamma}(1-e^{2\gamma})}{e^{\beta}(1-e^{2\beta})}\leq 2e^{\gamma-\beta}$. In particular, $r_1 \leq 2e^{-8} \leq \frac{1}{2}$. Similarly, $r_2 \leq 2\frac{\sinh s_1}{\sinh{2s_1}}=\frac{1}{\cosh s_1}\leq 2e^{-s_1}$. By Lemma \ref{lem:betagamma}, equation \eqref{eq:betagamma1}, we obtain that $r_2 \leq 2e^{-\frac{\beta}{4}}\leq 2e^{\frac{\gamma-\beta}{4}} \leq 2e^{-2} \leq \frac{1}{2}$. In particular, \eqref{eq:chialphaestimate} holds. Moreover, $|r_1-r_2| \leq \max\{r_1,r_2\} \leq 2e^{\frac{\gamma-\beta}{4}}$.

Because of the explicit form of $t^{\prime}$ and $t^{\prime\prime}$, we have $|t^{\prime}-t^{\prime\prime}| \leq \frac{\pi}{2}$. It follows that
\begin{equation} \nonumber
  |\dot{\varphi}(\beta,\gamma,t^{\prime})-\dot{\varphi}(2s_2,s_2,t^{\prime})| \leq |\dot{\varphi}(\beta,\gamma,t^{\prime})-\dot{\varphi}(2s_1,s_1,t^{\prime\prime})| + |\dot{\varphi}(2s_1,s_1,t^{\prime})-\dot{\varphi}(2s_1,s_1,t^{\prime\prime})|.
\end{equation}
The first summand is estimated by $4\sqrt{2}e^{-\frac{\beta-\gamma}{8}}\|\varphi\|_{M_0A(\widetilde{G})}$ by \eqref{eq:chialphaestimate}, and the second summand is estimated by $\tilde{B}e^{-\frac{\alpha}{4}}\|\varphi\|_{M_0A(\widetilde{G})}$ by Lemma \ref{lem:tdependences}. It follows that
\begin{equation} \nonumber
\begin{split}
  |\dot{\varphi}(\beta,\gamma,t^{\prime})-\dot{\varphi}(2s_2,s_2,t^{\prime})| &\leq 4\sqrt{2}e^{-\frac{\beta-\gamma}{8}}\|\varphi\|_{M_0A(\widetilde{G})} + \tilde{B}e^{-\frac{\alpha}{4}}\|\varphi\|_{M_0A(\widetilde{G})} \\
    &\leq (4\sqrt{2}+\tilde{B})e^{-\frac{\beta-\gamma}{8}}\|\varphi\|_{M_0A(\widetilde{G})}
\end{split}
\end{equation}
under the assumption that $\beta\geq\gamma+8$. By shifting over $t-t^{\prime}$ (cf.~Lemma \ref{lem:cinvariance}), we obtain the estimate of the lemma for $\beta \geq \gamma + 8$. In general, the assertion of the lemma follows with $B_1=\max\{4\sqrt{2}+\tilde{B},2e^2\}=4\sqrt{2}+\tilde{B}$.
\end{proof}
\begin{lem} \label{lem:betagammahyp}
  There exists a constant $B_2 > 0$ such that whenever $\beta \geq \gamma \geq 0$, $t \in \bbR$, and $s_2=s_2(\beta,\gamma)$ is chosen as in Lemma \ref{lem:betagamma}, then for all $\varphi \in M_0A(\widetilde{G}) \cap \mathcal{C}$,
\begin{equation} \nonumber
  |\dot{\varphi}(\beta,\gamma,t)-\dot{\varphi}(2s_2,s_2,t)| \leq B_2e^{-\frac{\gamma}{8}}\|\varphi\|_{M_0A(\widetilde{G})}.
\end{equation}
\end{lem}
\begin{proof}
Let $\beta \geq \gamma \geq 0$ and $t \in \mathbb{R}$. Assume first that $\gamma \geq 2$, and let $\alpha \in [0,\infty)$ be the unique solution in $[0,\infty)$ to the equation $\sinh\beta\sinh\gamma=\frac{1}{2}\sinh^2\alpha$, and observe that $\alpha>0$, because $\beta \geq \gamma \geq 2$. Put
\[
  a_1=\frac{\sinh\beta-\sinh\gamma}{\sinh (2\alpha)} \geq 0.
\]
Since $\sinh (2\alpha)=2\sinh \alpha \cosh \alpha \geq 2\sinh^2 \alpha$, we have
\[
  a_1 \leq \frac{\sinh \beta}{\sinh(2\alpha)} \leq \frac{\sinh \beta}{2\sinh^2\alpha}=\frac{1}{4\sinh\gamma}.
\]
In particular, $a_1 \leq \frac{1}{4\gamma} \leq \frac{1}{8}$. Put now $b_1=\sqrt{\frac{1}{2}-a_1^2}$. Then $1-a_1^2-b_1^2=\frac{1}{2}$. Hence, we have $\sinh \beta - \sinh \gamma=\sinh(2\alpha)a_1$. Let $h_1 \in H$ be such that
\[
  \iota(h_1)=\left(\begin{array}{cc} a_1+ib_1 & -\frac{1}{\sqrt{2}} \\ \frac{1}{\sqrt{2}} & a_1-ib_1 \end{array}\right) \in \SU(2).
\]
By Lemma \ref{lem:hyperbolaseqs}, we have $\widetilde{D}(\alpha,0)\widetilde{h_1}\widetilde{D}(\alpha,0) \in S_{\beta,\gamma,t^{\prime}}$, where $t^{\prime}$ is determined by the equations in that lemma. By Lemma \ref{lem:betagamma}, we have $\sinh(2s_2)\sinh s_2=\sinh \beta \sinh \gamma = \frac{1}{2}\sinh^2 \alpha$. Moreover, by \eqref{eq:betagamma1}, we have $s_2 \geq \frac{\gamma}{2} \geq 1$. By replacing $(\beta,\gamma)$ in the above computation with $(2s_2,s_2)$, we get that the number
\[
  a_2=\frac{\sinh (2s_2)-\sinh s_2}{\sinh (2\alpha)} \geq 0
\]
satisfies
\[
  a_2 \leq \frac{1}{4\sinh s_2} \leq \frac{1}{4\sinh 1} \leq \frac{1}{4}.
\]
Hence, we can put $b_2=\sqrt{\frac{1}{2}-a_2^2}$ and let $h_2 \in H$ be such that
\[
  \iota(h_2)=\left(\begin{array}{cc} a_2+ib_2 & -\frac{1}{\sqrt{2}} \\ \frac{1}{\sqrt{2}} & a_2-ib_2 \end{array}\right).
\]
Then
\begin{equation} \nonumber
\begin{split}
  \sinh(2s_2)\sinh s_2 &= \sinh^2 \alpha(1-a_2^2-b_2^2),\\
  \sinh(2s_2)-\sinh s_2 &= \sinh(2\alpha)a_2,
\end{split}
\end{equation}
and $\iota(h_2) \in \SU(2)$. Hence, by Lemma \ref{lem:hyperbolaseqs}, $\widetilde{D}(\alpha,0)\widetilde{h}_2\widetilde{D}(\alpha,0) \in S_{2s_2,s_2,t^{\prime\prime}}$, where $t^{\prime\prime}$ is determined by the equations in that lemma. It follows from the explicit formula for $t$, and from the fact that $t^{\prime}$ and $t^{\prime\prime}$ have the same sign, that $|t^{\prime}-t^{\prime\prime}| \leq \frac{\pi}{2}$. Put now $\theta_j=\mathrm{arg}(a_j+ib_j)=\frac{\pi}{2}-\sin^{-1} \left(\frac{a_j}{\sqrt{2}}\right)$ for $j=1,2$. Since $0 \leq a_j \leq \frac{1}{2}$ for $j=1,2$, and since $\frac{d}{dy} \sin^{-1} y = \frac{1}{\sqrt{1-y^2}}\leq\sqrt{2}$ for $y \in [0,\frac{1}{\sqrt{2}}]$, it follows that
\begin{equation} \nonumber
\begin{split}  
|\theta_1-\theta_2| &\leq \bigg\vert\sin^{-1} \left(\frac{a_1}{\sqrt{2}}\right) - \sin^{-1} \left(\frac{a_2}{\sqrt{2}}\right)\bigg\vert \\
  &\leq |a_1-a_2| \\
  &\leq \max\{a_1,a_2\} \\
  &\leq \max\left\{\frac{1}{4\sinh \gamma},\frac{1}{4\sinh t}\right\} \\
  &\leq \frac{1}{4\sinh \frac{\gamma}{2}},
\end{split}
\end{equation}
because $y \geq \frac{\gamma}{2}$. Since $\gamma \geq 2$, we have $\sinh \frac{\gamma}{2}=\frac{1}{2}e^{\frac{\gamma}{2}}(1-e^{-\gamma})\geq\frac{1}{4}e^{\frac{\gamma}{2}}$. Hence, $|\theta_1-\theta_2| \leq e^{-\frac{\gamma}{2}}$. Note that $a_j=\frac{1}{\sqrt{2}}e^{i\theta_j}$ for $j=1,2$, so by Proposition \ref{prp:psihyp}, we have
\begin{equation} \label{eq:psialpha}
\begin{split}
  |\dot{\varphi}(2s_2,s_2,t^{\prime\prime})-\dot{\varphi}(\beta,\gamma,t^{\prime})| &\leq |\psi_{\alpha}(h_1)-\psi_{\alpha}(h_2)| \\
    &\leq \tilde{C}|\theta_1-\theta_2|^{\frac{1}{4}}\|\psi_{\alpha}\|_{M_0A(H)} \\
    &\leq \tilde{C}e^{-\frac{\gamma}{8}}\|\varphi\|_{M_0A(\widetilde{G})}.
\end{split}
\end{equation}
Since $\widetilde{D}(\alpha,0)\widetilde{h}_1\widetilde{D}(\alpha,0) \in S_{\beta,\gamma,t^{\prime}}$ and $\widetilde{D}(\alpha,0)\widetilde{h}_2\widetilde{D}(\alpha,0) \in S_{2s_2,s_2,t^{\prime\prime}}$, it follows that
\begin{equation} \nonumber
  |\dot{\varphi}(\beta,\gamma,t^{\prime})-\dot{\varphi}(2s_2,s_2,t^{\prime})| \leq |\dot{\varphi}(\beta,\gamma,t^{\prime})-\dot{\varphi}(2s_2,s_2,t^{\prime\prime})| + |\dot{\varphi}(2s_1,s_1,t^{\prime})-\dot{\varphi}(2s_2,s_2,t^{\prime\prime})|.
\end{equation}
The first summand is estimated by $\tilde{C}e^{-\frac{\gamma}{8}}\|\varphi\|_{M_0A(\widetilde{G})}$ by \eqref{eq:psialpha}, and the second summand is estimated by $\tilde{B}e^{-\frac{\alpha}{4}}\|\varphi\|_{M_0A(\widetilde{G})}$. It now follows that
\begin{equation} \nonumber
  |\dot{\varphi}(\beta,\gamma,t^{\prime})-\dot{\varphi}(2s_1,s_1,t^{\prime})| \leq \tilde{C}e^{-\frac{\gamma}{8}}\|\varphi\|_{M_0A(\widetilde{G})} + \tilde{B}e^{-\frac{\alpha}{4}}\|\varphi\|_{M_0A(\widetilde{G})}.
\end{equation}
Using the fact that $e^{-\frac{\alpha}{4}} \leq e^{-\frac{\gamma}{8}}$ and using the invariance property of Lemma \ref{lem:cinvariance}, the desired estimate follows with $B_2=\max\{\tilde{C}+\tilde{B},2e^{\frac{1}{4}}\}$.
\end{proof}
We state the following lemma. For a proof, see \cite[Lemma 3.19]{haagerupdelaat1}.
\begin{lem} \label{lem:rhosigma}
  Let $s_1 \geq s_2 \geq 0$. Then the equations
\begin{equation} \nonumber
\begin{split}
  \sinh^2 \beta + \sinh^2 \gamma &= \sinh^2(2s_1)+\sinh^2 s_1, \\
  \sinh \beta \sinh \gamma &= \sinh(2s_2)\sinh s_2,
\end{split}
\end{equation}
have a unique solution $(\beta,\gamma) \in \mathbb{R}^2$ for which $\beta \geq \gamma \geq 0$. Moreover, if $1 \leq s_2 \leq s_1 \leq \frac{3}{2}s_2$, then
\begin{equation} \label{eq:system1}
\begin{split}
  |\beta-2s_1| &\leq 1, \\
  |\gamma+2s_1-3s_2| &\leq 1.
\end{split}
\end{equation}
\end{lem}
\begin{lem} \label{lem:comparest}
  There exists a constant $B_3 > 0$ such that whenever $s_1,s_2 \geq 0$ satisfy $2 \leq s_2 \leq s_1 \leq \frac{6}{5}s_2$ and $t \in \mathbb{R}$, then for all $\varphi \in M_0A(\widetilde{G}) \cap \mathcal{C}$,
\[
  |\dot{\varphi}(2s_1,s_1,t)-\dot{\varphi}(2s_2,s_2,t)| \leq B_3 e^{-\frac{s_1}{16}} \|\varphi\|_{M_0A(\widetilde{G})}.
\]
\end{lem}
\begin{proof}
  Choose $\beta \geq \gamma \geq 0$ as in Lemma \ref{lem:rhosigma}. Then by Lemma \ref{lem:betagammacir} and Lemma \ref{lem:betagammahyp}, we have
\begin{equation} \nonumber
\begin{split}
  |\dot{\varphi}(\beta,\gamma,t)-\dot{\varphi}(2s_1,s_1,t)| &\leq B_1e^{-\frac{\beta-\gamma}{8}}\|\varphi\|_{M_0A(\widetilde{G})} \\
  |\dot{\varphi}(\beta,\gamma,t)-\dot{\varphi}(2s_2,s_2,t)| &\leq B_2e^{-\frac{\gamma}{8}}\|\varphi\|_{M_0A(\widetilde{G})}.
\end{split}
\end{equation}
Moreover, by \eqref{eq:system1}, we have
\begin{equation} \nonumber
\begin{split}
  \beta - \gamma &\geq (2s_1-1) - (3s_2-2s_1+1) = 4s_1 - 3s_2 -2 \geq s_1-2, \\
  \gamma &\geq 3s_2-2s_1-1 \geq \frac{5}{2}s_1-2s_1-1=\frac{s_1-2}{2}.
\end{split}
\end{equation}
Hence, since $s_1 \geq 2$, we have $\min\{e^{-\gamma},e^{-(\beta-\gamma)}\} \leq e^{-\frac{s_1-2}{2}}$. Thus, the lemma follows from Lemma \ref{lem:betagammacir} and Lemma \ref{lem:betagammahyp} with $B_3=e^{\frac{1}{8}}(B_1+B_2)$.
\end{proof}
\begin{lem} \label{lem:limit}
  There exists a constant $B_4 > 0$ such that for all $\varphi \in M_0A(\widetilde{G}) \cap \mathcal{C}$ and $t \in \mathbb{R}$ the limit $c_{\varphi}(t)=\lim_{s_1 \to \infty} \dot{\varphi}(2s_1,s_1,t)$ exists, and for all $s_2 \geq 0$,
\[
  |\dot{\varphi}(2s_2,s_2,t)-c_{\varphi}(t)| \leq B_4e^{-\frac{s_2}{16}}\|\varphi\|_{M_0A(\widetilde{G})}.
\]
\end{lem}
\begin{proof}
  Let $\varphi \in M_0A(\widetilde{G}) \cap \mathcal{C}$, and let $t \in \mathbb{R}$. By Lemma \ref{lem:comparest}, we have for $u \geq 5$ and $\kappa \in [0,1]$, that
\begin{equation} \label{eq:ukappa}
  |\dot{\varphi}(2u,u,t)-\dot{\varphi}(2(u+\kappa),u+\kappa,t)| \leq B_3e^{-\frac{u}{16}}\|\varphi\|_{M_0A(\widetilde{G})}.
\end{equation}
Let $s_1 \geq s_2 \geq 5$. Then $s_1=s_2+n+\delta$, where $n \geq 0$ is an integer and $\delta \in [0,1)$. Applying equation \eqref{eq:ukappa} to $(u,\kappa)=(s_2+j,1)$, $j=0,1,\ldots,n-1$ and $(u,\kappa)=(s_2+n,\delta)$, we obtain
\[
  |\dot{\varphi}(2s_1,s_1,t)-\dot{\varphi}(2s_2,s_2,t)| \leq B_3\left(\sum_{j=0}^n e^{-\frac{s_2+j}{16}}\right) \|\varphi\|_{M_0A(\widetilde{G})} \leq B_3^{\prime}e^{-\frac{s_2}{16}}\|\varphi\|_{M_0A(\widetilde{G})},
\]
where $B_3^{\prime}=(1-e^{-\frac{1}{16}})^{-1}B_3$. Hence $(\dot{\varphi}(2s_1,s_1,t))_{s_1 \geq 5}$ is a Cauchy net for every $t \in \mathbb{R}$. Therefore, $c_{\varphi}(t)=\lim_{s_1 \to \infty} \dot{\varphi}(2s_1,s_1,t)$ exists, and
\[
  |\dot{\varphi}(2s_2,s_2,t)-c_{\varphi}(t)|=\lim_{s_1 \to \infty} |\dot{\varphi}(2s_1,s_1,t)-\dot{\varphi}(2s_2,s_2,t)| \leq B_3^{\prime}e^{-\frac{s_2}{16}}\|\varphi\|_{M_0A(\widetilde{G})}
\]
for all $s_2 \geq 5$. Since $\|\varphi\|_{\infty} \leq \|\varphi\|_{M_0A(\widetilde{G})}$, we have for all $0 \leq s_2 < 5$,
\[
  |\dot{\varphi}(2s_2,s_2,t)-c_{\varphi}(t)| \leq 2\|\varphi\|_{M_0A(\widetilde{G})}.
\]
Hence, the lemma follows with $B_4=\max\{B_3^{\prime},2e^{\frac{5}{16}}\}$.
\end{proof}
\begin{proof}[Proof of Proposition \ref{prp:sp2ab}]
Let $\varphi \in M_0A(\widetilde{G}) \cap \mathcal{C}$ and let $t \in \mathbb{R}$. Let $\beta \geq \gamma \geq 0$. Suppose first that $\beta \geq 2\gamma$. Then $\beta - \gamma \geq \frac{\beta}{2}$, so by Lemma \ref{lem:betagamma} and Lemma \ref{lem:betagammacir}, there exists an $s_1 \geq \frac{\beta}{4}$ such that
\[
  |\dot{\varphi}(\beta,\gamma,t)-\dot{\varphi}(2s_1,s_1,t)| \leq B_1e^{-\frac{\beta}{16}}\|\varphi\|_{M_0A(\widetilde{G})}.
\]
Suppose now that $\beta < 2\gamma$. Then, by Lemma \ref{lem:betagamma} and Lemma \ref{lem:betagammahyp}, we obtain that there exists an $s_2 \geq \frac{\gamma}{2} > \frac{\beta}{4}$ such that
\[
  |\dot{\varphi}(\beta,\gamma,t)-\dot{\varphi}(2s_2,s_2,t)| \leq B_2e^{-\frac{\beta}{16}}\|\varphi\|_{M_0A(\widetilde{G})}.
\]
Combining these estimates with Lemma \ref{lem:limit}, and using again that $s_1$ and $s_2$ majorize $\frac{\beta}{4}$, it follows that for all $\beta \geq \gamma \geq 0$, we have
\[
  |\dot{\varphi}(\beta,\gamma,t)-c_{\varphi}(t)| \leq C_1e^{-\frac{\beta}{64}}\|\varphi\|_{M_0A(\widetilde{G})}.
\]
where $C_1=\max\{B_1+B_4,B_2+B_4\}$. This proves the proposition, for $\sqrt{\beta^2+\gamma^2} \leq \sqrt{2}\beta$.
\end{proof}
\begin{prp}
  For every $\varphi \in M_0A(\widetilde{G}) \cap \mathcal{C}$, the limit function $c_{\varphi}(t)$ is a constant function.
\end{prp}
\begin{proof}
  From Proposition \ref{prp:sp2ab} and its proof, we know that for every $\varphi \in M_0A(\widetilde{G})$ the limit $c_{\varphi}(t)=\lim_{\beta^2+\gamma^2 \to \infty} \dot{\varphi}(\beta,\gamma,t)$ exists and that $\varphi$ satisfies a certain asymptotic behaviour. It is clear from this expression that the limit may depend on $t$, but it does not depend on how $\beta^2+\gamma^2$ goes to infinity. In particular, we have $c_{\varphi}(t)=\lim_{\alpha \to \infty} \dot{\varphi}(2\alpha,0,t)$. Let $\tau_1,\tau_2$ be such that $|\tau_1-\tau_2| \leq \frac{\pi}{2}$. By Proposition \ref{prp:tdependence}, we have
\[
  |\dot{\varphi}(2\alpha,0,\tau_1)-\dot{\varphi}(2\alpha,0,\tau_2)| \leq 24e^{-2\alpha}\|\varphi\|_{M_0A(\widetilde{G})}.
\]
In the limit $\alpha \to \infty$, this expression gives $c_{\varphi}(\tau_1)=c_{\varphi}(\tau_2)$. i.e., the function $c_{\varphi}(t)$ is constant on any interval of length smaller than or equal to $\frac{\pi}{2}$. Hence, the function $c_{\varphi}$ is constant.
\end{proof}
\begin{cor}
  The space $M_0A(\widetilde{G}) \cap \mathcal{C}_0$ of completely bounded Fourier multipliers $\varphi$ in $\mathcal{C}$ for which $c_{\varphi} \equiv 0$ is a subspace of $M_0A(\widetilde{G}) \cap \mathcal{C}$ of codimension one.
\end{cor}

\section{Noncommutative $L^p$-spaces associated with lattices in $\widetilde{\mathrm{Sp}}(2,\mathbb{R})$} \label{sec:nclpspaces}
Let again $G=\mathrm{Sp}(2,\mathbb{R})$ and $\widetilde{G}=\widetilde{\mathrm{Sp}}(2,\mathbb{R})$. We use the same realization of $\widetilde{G}$ as in Section \ref{sec:covsp2r}, and we use the same notation as in that section (e.g., for the subgroups $K$, $A$, $\overline{{A}^{+}}$, $H$, $H_0$, $H_1$ of $G$ and the corresponding subgroups of $\widetilde{G}$). The main result of this section is a statement about the $\apschur$ for $\widetilde{G}$. This gives rise to the failure of the OAP for certain noncommutative $L^p$-spaces, which will be explained in Section \ref{sec:mainresults}.
\begin{thm} \label{thm:covsp2noapschur}
  For $p \in [1,\frac{12}{11}) \cup (12,\infty]$, the group $\widetilde{G}$ does not have the $\apschur$.
\end{thm}
The proof follows by combining the method of proof of the failure of the AP for $\widetilde{\mathrm{Sp}}(2,\bbR)$ in Section \ref{sec:covsp2r} with the methods that were used in \cite{ldls}, \cite{delaat1} to prove the failure of the $\apschur$ for $\mathrm{SL}(3,\mathbb{R})$ and $\mathrm{Sp}(2,\mathbb{R})$ for certain values of $p \in (1,\infty)$, respectively.

Note that for $p=1$ and $\infty$, the $\apschur$ is equivalent to weak amenability (see \cite[Proposition 2.3]{ldls}), and the failure of weak amenability for $\widetilde{G}$ was proved in \cite{dorofaeff}, so from now on, it suffices to consider $p \in (1,\infty)$. Using an averaging argument similar to the one in Lemma \ref{lem:restrictiontoc} (see \cite[Lemma 2.8]{delaat1} for more details on averaging functions in the setting of the $\apschur$), it follows that if $\widetilde{G}$ has the $\apschur$ for some $p \in (1,\infty)$, then the approximating net can be chosen in $A(\widetilde{G}) \cap \mathcal{C}$.

The following result, which is a direct analogue of Proposition \ref{prp:sp2ab}, gives a certain asymptotic behaviour of continuous functions $\varphi$ in $\mathcal{C}$ for which the induced function $\check{\varphi}$ is a Schur multiplier on $S^p(L^2(\widetilde{G}))$. From this, it follows that the constant function $1$ cannot be approximated pointwise (and hence not uniformly on compacta) by a net in $A(\widetilde{G}) \cap \mathcal{C}$ in such a way that the net of associated multipliers is uniformly bounded in the $MS^p(L^2(\widetilde{G}))$-norm. This implies Theorem \ref{thm:covsp2noapschur}.
\begin{prp} \label{prp:sp2abapschur}
  Let $p > 12$. There exist constants $C_1(p),C_2(p)$ (depending on $p$ only) such that for all $\varphi \in C(\widetilde{G}) \cap \mathcal{C}$ for which $\check{\varphi} \in MS^p(L^2(\widetilde{G}))$, and for all $t \in \mathbb{R}$, the limit $\tilde{c}_{\varphi}^p(t) = \lim_{s \to \infty} \dot{\varphi}(2s,s,t)$ exists, and for all $\beta \geq \gamma \geq 0$,
\[
  |\dot{\varphi}(\beta,\gamma,t)-\tilde{c}_{\varphi}^p(t)| \leq C_1(p)e^{-C_2(p)\sqrt{\beta^2 + \gamma^2}}\|\check{\varphi}\|_{MS^p(L^2(\widetilde{G}))}.
\]
\end{prp}
To prove this, we again use the strong Gelfand pair $(\SU(2),\U(1))$ and the Gelfand pair $(\SU(2),\SO(2))$, which sit inside $\widetilde{G}$. For the disc polynomials $h_{l,m}$, we need better estimates than in Lemma \ref{lem:hoelderhpq}. These were already given in \cite[Corollary 3.5]{haagerupdelaat1}.
\begin{lem} \label{lem:hoelderu2p}
  For all $l,m \geq 0$, and for $\theta_1,\theta_2 \in [0,2\pi)$, we have
\[
  \biggl\vert h_{l,m}^0 \left( \frac{e^{i\theta_1}}{\sqrt{2}} \right) - h_{l,m}^0 \left( \frac{e^{i\theta_2}}{\sqrt{2}} \right) \biggr\vert \leq C(l+m+1)^{\frac{3}{4}}|\theta_1-\theta_2|,
\]
\[
  \biggl\vert h_{l,m}^0 \left( \frac{e^{i\theta_1}}{\sqrt{2}} \right) - h_{l,m}^0 \left( \frac{e^{i\theta_2}}{\sqrt{2}} \right) \biggr\vert \leq 2C(l+m+1)^{-\frac{1}{4}}.
\]
Here $C>0$ is a uniform constant.
\end{lem}
Combining the above two estimates, we get the estimate of Lemma \ref{lem:hoelderhpq}. Combining Lemma \ref{lem:hoeldersu2u1} and \cite[Lemma 2.4]{delaat1}, we obtain that for $\varphi \in L^2(\SU(2) // \U(1))$, there is an induced function $\varphi^0:\mathbb{D} \rightarrow \mathbb{C}$, and
\[
  \varphi^0=\sum_{l,m=0}^{\infty} c_{l,m}(l+m+1) h_{l,m}^0
\]
for certain $c_{l,m} \in \bbC$. Moreover, by \cite[Proposition 2.7]{delaat1}, we obtain that if $p \in (1,\infty)$, then $(\sum_{l,m \geq 0} |c_{l,m}|^p (l+m+1))^{\frac{1}{p}} \leq \|\check{\varphi}\|_{MS^p(L^2(\U(2)))}$.
\begin{lem} \label{lem:behavioru2p}
Let $p > 12$, and let $\varphi : \SU(2) \rightarrow \mathbb{C}$ be a continuous $\mathrm{Int}(\U(1))$-invariant function such that $\check{\varphi}$ is an element of $MS^p(L^2(\SU(2)))$. Then $\varphi^0$ satisfies
\[
  \left|\varphi^0\left(\frac{e^{i\theta_1}}{\sqrt{2}}\right)-\varphi^0\left(\frac{e^{i\theta_2}}{\sqrt{2}}\right)\right| \leq \tilde{C}(p)\|{\check{\varphi}}\|_{MS^p(L^2(\U(2)))}|\theta_1-\theta_2|^{\frac{1}{8}-\frac{3}{2p}}
\]
for $\theta_1,\theta_2 \in [0,2\pi)$. Here, $\tilde{C}(p)$ is a constant depending only on $p$.
\end{lem}
The proof of this lemma is exactly the same as the proof of \cite[Lemma 3.5]{delaat1} after identifying the spaces $C(\SU(2) // \U(1))$ and $C(L \backslash \U(2) / L)$ and proving an isometric isomorphism in the setting of multipliers on Schatten classes as was done in Lemma \ref{lem:hoeldersu2u1} in the setting of completely bounded Fourier multipliers.
\begin{lem} \label{lem:fromGtoKp1}
	Let $\varphi \in C(\widetilde{G}) \cap \mathcal{C}$ such that $\check{\varphi} \in MS^p(L^2(\widetilde{G}))$ for some $p \in (1,\infty)$, and for $\alpha \in \bbR$, let $\psi_{\alpha}:H \rightarrow \bbC$ be defined by $\psi_{\alpha}(h)=\varphi(\widetilde{D}(\alpha,0)\widetilde{h}\widetilde{D}(\alpha,0))$. Then $\psi_{\alpha}$ is an element of $C(H // H_0)$ and satisfies
	\[
		\|\check{\psi}_{\alpha}\|_{MS^p(L^2(H))} \leq \|\check{\varphi}\|_{MS^p(L^2(\widetilde{G}))}.
	\]
\end{lem}
\begin{proof}
The fact that $\psi_{\alpha} \in C(H // H_0)$ follows as in Proposition \ref{prp:psihyp}. The second part follows by the fact that $\widetilde{D}(\alpha,0)\widetilde{H}\widetilde{D}(\alpha,0)$ is a subset of $\widetilde{G}$ and by applying \cite[Lemma 2.3]{delaat1}.
\end{proof}
We now turn to the second pair of groups $(H,H_1)$. We again need the Legendre polynomials, which act as spherical functions. The following estimate was proved in \cite[Lemma 3.8]{delaat1}.
\begin{lem} \label{lem:hoeldersu2p}
For all non-negative integers $n$, and $x,y \in [-\frac{1}{2},\frac{1}{2}]$,
\begin{equation} \nonumber
\begin{split}
  |P_n(x)-P_n(y)| &\leq |P_n(x)|+|P_n(y)| \leq \frac{4}{\sqrt{n}},\\
  |P_n(x)-P_n(y)| &\leq \left|\int_x^y P_n^{\prime}(t)dt\right| \leq 4\sqrt{n}|x-y|.
\end{split}
\end{equation}
\end{lem}
Combining the two estimates above, yields the estimate of Lemma \ref{lem:legendreestimates}. Let $\varphi:\SU(2) \rightarrow \bbC$ be a $\SO(2)$-bi-invariant continuous function. Then $\varphi(h)=\varphi^0(r)$ as in Section \ref{sec:covsp2r}. It follows that $\varphi^0=\sum_{n=0}^{\infty} c_n(2n+1) P_n$ for certain $c_n \in \bbC$. Moreover, as above, we obtain that if $p \in (1,\infty)$, then $(\sum_{n \geq 0} |c_n|^p (2n+1))^{\frac{1}{p}} \leq \|\check{\varphi}\|_{MS^p(L^2(\SU(2)))}$, where $\check{\varphi}$ is defined as above by $\check{\varphi}(g,h)=\varphi(g^{-1}h)$. The following result can be found in \cite[Lemma 3.9]{delaat1}.
\begin{lem} \label{lem:behaviorsu2p}
Let $p>4$, and let $\varphi \in C(\SO(2) \backslash \SU(2) / \SO(2))$ be such that $\check{\varphi} \in MS^p(L^2(\SU(2)))$. Then $\varphi^0$ satisfies
\[
  |\varphi^0(\delta_1)-\varphi^0(\delta_2)| \leq \hat{C}(p)\|{\varphi}\|_{MS^p(L^2(\SU(2)))}|\delta_1-\delta_2|^{\frac{1}{4}-\frac{1}{p}}
\]
for $\delta_1,\delta_2 \in [-\frac{1}{2},\frac{1}{2}]$. Here $\hat{C}(p)$ is a constant depending only on $p$.
\end{lem}
\begin{lem}
Let $\varphi \in C(\widetilde{G}) \cap \mathcal{C}$ such that $\check{\varphi} \in MS^p(L^2(\widetilde{G}))$ for some $p \in (1,\infty)$. For $\alpha \geq 0$, let $\chi_{\alpha}^{\prime}:H \rightarrow \mathbb{C}$ be defined by $h \mapsto \varphi(\widetilde{D}(\alpha,\alpha)\widetilde{v}\widetilde{h}\widetilde{D}(\alpha,\alpha))$, and let $\chi_{\alpha}^{\prime\prime}:H \rightarrow \mathbb{C}$ bedefined by $h \mapsto \varphi(\widetilde{D}(\alpha,\alpha)\widetilde{v}^{-1}\widetilde{h}\widetilde{D}(\alpha,\alpha))$. These maps are $H_1$-bi-invariant such that $\check{\chi}_{\alpha}^{\prime},\check{\chi}_{\alpha}^{\prime\prime} \in MS^p(L^2(H))$. Moreover, we obtain that $\|\check{\chi}_{\alpha}^{\prime}\|_{MS^p(L^2(H))} \leq \|\check{\varphi}\|_{MS^p(L^2(\widetilde{G}))}$ and $\|\check{\chi}_{\alpha}^{\prime\prime}\|_{MS^p(L^2(H))} \leq \|\check{\varphi}\|_{MS^p(L^2(\widetilde{G}))}$. 
\end{lem}
The fact that the maps are $H_1$-bi-invariant is similar to the case of completely bounded Fourier multipliers. The second part follows by the fact that the sets $\widetilde{D}(\alpha,\alpha)\widetilde{v}\widetilde{H}\widetilde{D}(\alpha,\alpha)$ and $\widetilde{D}(\alpha,\alpha)\widetilde{v}^{-1}\widetilde{H}\widetilde{D}(\alpha,\alpha)$ are subsets of $\widetilde{G}$ and by applying \cite[Lemma 2.3]{delaat1}.
\begin{prp} \label{prp:tdependencep}
  Let $p > 4$, and let $\varphi \in \mathcal{C}$ such that $\check{\varphi} \in MS^p(L^2(\widetilde{G}))$. If $|\tau_1-\tau_2| \leq \frac{\pi}{2}$ and $\alpha \geq 0$, then
\[
  |\dot{\varphi}(2\alpha,0,\tau_1)-\dot{\varphi}(2\alpha,0,\tau_2)| \leq D(p) e^{-2\alpha(\frac{1}{4}-\frac{1}{p})} \|\check{\varphi}\|_{MS^p(L^2(\widetilde{G}))},
\]
where $D(p) > 0$ is a constant depending only on $p$.
\end{prp}
The proof of this proposition is similar to the proof of Proposition \ref{prp:tdependence}. One uses the H\"older continuity coming from the Legendre polynomials in the $p$-setting (see Lemma \ref{lem:behaviorsu2p}) rather than the H\"older continuity in the setting of completely bounded Fourier multipliers. In the lemma yielding the above proposition, replace the H\"older continuity accordingly.
\begin{lem} \label{lem:tdependencesp}
  There exists a constant $\tilde{B}(p) > 0$ such that for $\alpha > 0$, $t \in \mathbb{R}$, $\tau \in [-\frac{\pi}{2},\frac{\pi}{2}]$, and $s_1=s_1(2\alpha,0)$ is chosen as in Lemma \ref{lem:betagamma}, then for all $\varphi \in \mathcal{C}$ such that $\check{\varphi} \in MS^p(L^2(\widetilde{G}))$,
\[
  |\dot{\varphi}(2s_1,s_1,t) - \dot{\varphi}(2s_1,s_1,t+\tau)| \leq \tilde{B}(p)e^{-\frac{\alpha}{2}(\frac{1}{4}-\frac{1}{p})} \|\check{\varphi}\|_{MS^p(L^2(\widetilde{G}))}.
\]
\end{lem}
The following two lemmas replace Lemmas \ref{lem:betagammacir} and \ref{lem:betagammahyp}.
\begin{lem} \label{lem:betagammacirp}
  For $p > 4$, there exists a constant $B_1(p) > 0$ (depending only on $p$) such that whenever $\beta \geq \gamma \geq 0$ and $s_1=s_1(\beta,\gamma)$ is chosen as in Lemma \ref{lem:betagamma}, then for all $\varphi \in \mathcal{C}$ for which $\check{\varphi} \in MS^p(L^2(\widetilde{G}))$,
\begin{equation} \nonumber
  |\dot{\varphi}(\beta,\gamma,t)-\dot{\varphi}(2s_1,s_1,t)| \leq B_1(p) e^{-\frac{\beta-\gamma}{4}(\frac{1}{4}-\frac{1}{p})} \|\check{\varphi}\|_{MS^p(L^2(\widetilde{G}))}.
\end{equation}
\end{lem}
\begin{lem} \label{lem:betagammahypp}
  For $p > 12$, there exists a constant $B_2(p) > 0$ (depending only on $p$) such that whenever $\beta \geq \gamma \geq 0$ and $s_2=s_2(\beta,\gamma)$ is chosen as in Lemma \ref{lem:betagamma}, then for all $\varphi \in \mathcal{C}$ for which $\check{\varphi} \in MS^p(L^2(\widetilde{G}))$,
\begin{equation} \nonumber
  |\dot{\varphi}(\beta,\gamma,t)-\dot{\varphi}(2s_2,s_2,t)| \leq B_2(p)e^{-\frac{\gamma}{4}(\frac{1}{4}-\frac{3}{p})}\|\check{\varphi}\|_{MS^p(L^2(\widetilde{G}))}.
\end{equation}
\end{lem}
The following lemma follows in a similar way from the previous two lemmas as Lemma \ref{lem:comparest} follows from Lemmas \ref{lem:betagammacir} and \ref{lem:betagammahyp}.
\begin{lem} \label{lem:comparestp}
  For all $p > 12$, there exists a constant $B_3(p) > 0$ such that whenever $s_1,s_2 \geq 0$ satisfy $2 \leq s_2 \leq s_1 \leq \frac{6}{5}s_2$, then for all $\varphi \in \mathcal{C}$ for which $\check{\varphi} \in MS^p(L^2(\widetilde{G}))$ and for all $t \in \mathbb{R}$,
\[
  |\dot{\varphi}(2s_1,s_1,t)-\dot{\varphi}(2s_2,s_2,t)| \leq B_3(p) e^{-\frac{s_1}{8}(\frac{1}{4}-\frac{3}{p})} \|\check{\varphi}\|_{MS^p(L^2(\widetilde{G}))}.
\]
\end{lem}
The following lemma replaces \ref{lem:limit}.
\begin{lem} \label{lem:limitp}
  For $p > 12$, there exists a constant $B_4(p) > 0$ such that for all $\varphi \in \mathcal{C}$ for which $\check{\varphi} \in MS^p(L^2(\widetilde{G}))$ and for all $t \in \mathbb{R}$, the limit $\tilde{c}_{\varphi}^p(t)=\lim_{s_1 \to \infty} \dot{\varphi}(2s_1,s_1,t)$ exists, and for all $s_2 \geq 0$,
\[
  |\dot{\varphi}(2s_2,s_2,t)-\tilde{c}_{\varphi}^p(t)| \leq B_4(p)e^{-\frac{s_2}{8}(\frac{1}{4}-\frac{3}{p})}\|\check{\varphi}\|_{MS^p(L^2(\widetilde{G}))}.
\]
\end{lem}
\begin{proof}[Proof of Proposition \ref{prp:sp2abapschur}]
Let $\varphi \in \mathcal{C}$ be such that $\check{\varphi} \in MS^p(L^2(\widetilde{G}))$. The proof of the proposition now follows in the same way as the proof of Proposition \ref{prp:sp2ab}. Indeed, assume first $\beta \geq 2\gamma$. Then $\beta - \gamma \geq \frac{\beta}{2}$, and it follows for all $t \in \mathbb{R}$ that
\[
  |\dot{\varphi}(\beta,\gamma,t)-\tilde{c}_{\varphi}^p(t)| \leq (B_1(p)+B_4(p))e^{-\frac{\beta}{32}(\frac{1}{4}-\frac{3}{p})}\|\check{\varphi}\|_{MS^p(L^2(\widetilde{G}))}.
\]
Assume now that $\beta < 2\gamma$. Then
\[
  |\dot{\varphi}(\beta,\gamma,t)-\tilde{c}_{\varphi}^p(t)| \leq (B_2(p)+B_4(p))e^{-\frac{\beta}{32}(\frac{1}{4}-\frac{3}{p})}\|\check{\varphi}\|_{MS^p(L^2(\widetilde{G}))}.
\]
Combining these results, it follows that for all $\beta \geq \gamma \geq 0$,
\[
  |\dot{\varphi}(\beta,\gamma,t)-\tilde{c}_{\varphi}^p(t)| \leq C_1(p)e^{-C_2(p)\sqrt{\beta^2+\gamma^2}}\|\check{\varphi}\|_{MS^p(L^2(\widetilde{G}))},
\]
where $C_1(p)=\max\{B_1(p)+B_4(p),B_2(p)+B_4(p)\}$ and $C_2(p)=\frac{1}{32\sqrt{2}}(\frac{1}{4}-\frac{3}{p})$. This proves the proposition.
\end{proof}
The values $p \in [1,\frac{12}{11}) \cup (12,\infty]$ give sufficient conditions for $\widetilde{G}$ to fail the $\apschur$. We would like to point out that the set of these values might be bigger, as already mentioned in Section \ref{sec:introduction}.

\section{Main results} \label{sec:mainresults}
In this section, we state and prove the main results of this article.
\begin{thm} \label{thm:apctdsimple}
	Let $G$ be a connected simple Lie group. Then $G$ has the Approximation Property if and only if it has real rank zero or one.
\end{thm}
\begin{proof}
Since it is well-known that if a connected simple Lie group $G$ has real rank zero or one, then $G$ has the AP (see Section \ref{sec:introduction}), it suffices to prove that any connected simple Lie group with real rank greater than or equal to two does not have the AP.

Let $G$ be a connected simple Lie group with real rank greater than or equal to two. Then $G$ has a closed connected subgroup $H$ locally isomorphic to $\mathrm{SL}(3,\mathbb{R})$ or $\mathrm{Sp}(2,\mathbb{R})$ (see, e.g., \cite{boreltits},\cite{dorofaeff},\cite{margulis}).

Firstly, suppose that $H$ is locally isomorphic to $\mathrm{SL}(3,\mathbb{R})$. Since the universal covering $\widetilde{\mathrm{SL}}(3,\mathbb{R})$ has finite center, it follows that $H$ automatically has finite center. Using the fact that the AP is preserved under local isomorphism of connected simple Lie groups with finite center (see Section \ref{subsec:ap}) and the fact that $\mathrm{SL}(3,\mathbb{R})$ does not have the AP, it follows that $G$ does not have the AP, since the AP passes from a group to closed subgroups.

Secondly, suppose that $H$ is locally isomorphic to $\Sp(2,\bbR)$, i.e., $H$ is isomorphic to $\widetilde{\mathrm{Sp}}(2,\mathbb{R}) / \Gamma$, where $\Gamma$ is a discrete subgroup of the center $Z(\widetilde{\mathrm{Sp}}(2,\mathbb{R}))$ of $\widetilde{\mathrm{Sp}}(2,\mathbb{R})$. If $H$ has finite center, then the result follows in the same way as the case $\mathrm{SL}(3,\mathbb{R})$. If $H$ has infinite center, then $H \cong \widetilde{\mathrm{Sp}}(2,\mathbb{R})$, because all nontrivial subgroups of the center of $\widetilde{\mathrm{Sp}}(2,\mathbb{R})$ are infinite subgroups of finite index (which make $H$ have finite center). This implies that $H$ does not have the AP, which finishes the proof.
\end{proof}
Note that the proof of this theorem follows from combining the failure of the AP for $\mathrm{SL}(3,\mathbb{R})$, which was proved by Lafforgue and de la Salle and the failure of the AP for $\mathrm{Sp}(2,\mathbb{R})$ and $\widetilde{\mathrm{Sp}}(2,\mathbb{R})$.
\begin{cor}
  Let $G=S_1 \times \ldots \times S_n$ be a connected semisimple Lie group with connected simple factors $S_i$, $i=1,\ldots,n$. Then $G$ has the AP if and only if for all $i=1,\ldots,n$ the real rank of $S_i$ is smaller than or equal to $1$.
\end{cor}
We now state our results on noncommutative $L^p$-spaces. Combining \cite[Theorem E]{ldls} by Lafforgue and de la Salle, \cite[Theorem 3.1]{delaat1} and Theorem \ref{thm:covsp2noapschur} of this article, it follows that whenever $G$ is a connected simple Lie group with real rank greater than or equal to two and whenever $p \in [1,\frac{12}{11}) \cup (12,\infty]$, then $G$ does not have the $\apschur$. Combining this with the fact that the $\apschur$ passes from a group to its lattices and vice versa and the earlier mentioned result of Lafforgue and de la Salle that whenever $\Gamma$ is a discrete group such that $L^p(L(\Gamma))$ has the OAP for $p \in (1,\infty)$, then $\Gamma$ has the $\apschur$, we obtain the following result.
\begin{thm} \label{thm:nclpsmain}
  Let $\Gamma$ be a lattice in a connected simple Lie group with real rank greater than or equal to two. For $p \in [1,\frac{12}{11}) \cup (12,\infty]$, the noncommutative $L^p$-space $L^p(L(\Gamma))$ does not have the OAP or CBAP.
\end{thm}
Note that this result only gives sufficient conditions on the value of $p$ for the failure of the CBAP and OAP for noncommutative $L^p$-spaces associated with lattices in connected higher rank simple Lie groups. The set of such $p$-values might be bigger than $[1,\frac{12}{11}) \cup (12,\infty]$. In particular, if we consider $L^p(L(\Gamma))$, where $\Gamma$ is a lattice in a connected simple Lie group that contains a closed subgroup locally isomorphic to $\mathrm{SL}(3,\mathbb{R})$, then we know by the results of Lafforgue and de la Salle that the CBAP and OAP for $L^p(L(\Gamma))$ fail for $p \in [1,\frac{4}{3}) \cup (4,\infty]$.

\begin{appendix}

\section{Harmonic analysis on strong Gelfand pairs} \label{sec:sgp}
This appendix discusses the analogues of spherical functions in the setting of strong Gelfand pairs. In particular, we explain their relation to spherical functions for Gelfand pairs and their meaning in representation theory. The material discussed here is not needed for the rest of this article, but might give a deeper understanding of certain results proved in Sections \ref{sec:covsp2r} and \ref{sec:nclpspaces} (see in particular Lemma \ref{lem:hoeldersu2u1}). The main result of this section, Theorem \ref{thm:bijectivecorrespondencesphericalsspherical}, might be known to experts, and special cases of it were considered in \cite{flenstedjensen}, but we could not find a reference for the general statement. The content of this appendix arose from discussions between the second named author and Thomas Danielsen.

The definitions of Gelfand pairs, spherical functions and strong Gelfand pairs were given in Section \ref{subsec:gelfandpairs}. It was pointed out there (and it is elementary to prove) that a pair $(G,K)$ consisting of a locally compact group $G$ and a compact subgroup $K$ is a strong Gelfand pair if and only if $(G\times K, \Delta K)$ (where $\Delta K$ is the diagonal subgroup) is a Gelfand pair. We refer to \cite{vandijk} and \cite{faraut} for a thorough account of the theory of Gelfand pairs.

Suppose that $G$ is a locally compact group with compact subgroup $K$. An equivalent definition of spherical functions (see \cite{vandijk},\cite{faraut} for a proof of the equivalence) is that for a Gelfand pair $(G,K)$, a function $h \in C(K \backslash G / K)$ that is not identical to zero is spherical if for all $g_1,g_2 \in G$ we have $\int_K h(g_1kg_2)dk = h(g_1)h(g_2)$. We denote the set of spherical functions by $S(G,K)$. Spherical functions parametrize the nontrivial characters (multiplicative linear functionals) of the algebra $C_c(K\backslash G / K)$, since any such character is of the form $\chi(\varphi)=\chi_h(\varphi)=\int_G \varphi(g)h(g^{-1})dg$. Furthermore, if $h$ is a bounded spherical function, the expression above defines a continuous multiplicative functional on the Banach algebra $L^1(K\backslash G / K)$, and the set $BS(G,K)$ of bounded spherical functions parametrizes bijectively the set of continuous characters of $L^1(K\backslash G / K)$.

We can now define the analogues of spherical functions in the setting of strong Gelfand pairs. For a strong Gelfand pair $(G,K)$, we say that a function $h \in C(G // K)$ that is not identical to zero is s-spherical if for all $g_1,g_2 \in G$ we have $\int_K h(k^{-1}g_1kg_2)dk=h(g_1)h(g_2)$. The set of s-spherical functions is denoted by $SS(G,K)$. Analogous to the case of spherical functions, the s-spherical functions parametrise the space of nontrivial characters of the convolution algebra $C_c(G//K)$, since an s-spherical function $h$ gives rise to a character by $\chi_h(\varphi):=\int_G \varphi(g)h(g^{-1})dg$.

It is clear that $S(G,K)\subset SS(G,K)$. We can now relate the spaces of s-spherical functions for $(G,K)$ and spherical functions for $(G \times K,\Delta K)$. First, we state a lemma, the proof of which is elementary and left to the reader.
\begin{lem} \label{lem:bijectivecorrespondence}
The map $\Phi:\Delta K\backslash (G\times K) /\Delta K\to G//K$ given by $\Delta K (g,k) \Delta K \mapsto [k^{-1}g]=[gk^{-1}]$ is a homeomorphism with inverse $\Phi^{-1}([g])=\Delta K (g,e) \Delta K$. Here, $[g]$ denotes the $K$-conjugation class of $g$.
\end{lem}
The map $\Phi$ of Lemma \ref{lem:bijectivecorrespondence} induces a bijection $\Phi^*:C(G//K) \rightarrow C(K\backslash G / K)$ given by $f\mapsto f\circ \Phi$.
\begin{thm} \label{thm:bijectivecorrespondencesphericalsspherical}
The map $\Phi^*:C(G//K) \rightarrow C(K\backslash G / K)$ given by $f\mapsto f\circ \Phi$ defines a bijection between $SS(G,K)$ and $S(G\times K,\Delta K)$.
\end{thm}
\begin{proof}
For $h \in SS(G,K)$, we have $(h \circ \Phi)((k_1,k_1)(g,k)(k_2,k_2))=h(k_2^{-1}k^{-1}gk_2)=h(k^{-1}g)=(h \circ \Phi)((g,k))$ for all $g \in G$ and $k,k_1,k_2 \in K$, so $h \circ \Phi$ is $\Delta K$-bi-invariant on $G \times K$. Moreover, we check that for $h \circ \Phi$ (which is not identical to the zero function), we have
\begin{equation} \nonumber
\begin{split}
  \int_K (h \circ \Phi)((g_1,k_1)(k,k)(g_2,k_2))dk = \int_K h(k_2^{-1}k^{-1}k_1^{-1}g_1kg_2)dk \\
    = \int_K h(k^{-1}k_1^{-1}g_1kg_2k_2^{-1})dk = (h \circ \Phi)(g_1,k_1)(h \circ \Phi)(g_2,k_2)
\end{split}
\end{equation}
for all $g_1,g_2 \in G$ and $k,k_1,k_2 \in K$. Let now $h \in S(G,K)$. It follows that $(h \circ \Phi^{-1})(kgk^{-1})=h(kg,k)=h(g,e)=(h \circ \Phi^{-1})(g)$ for all $g \in G$ and $k \in K$, so $h \circ \Phi^{-1}$ is $\mathrm{Int}(K)$-invariant on $G$. Moreover, we check that $\varphi \circ \Phi^{-1}$ (which is not identical to the zero function) satisfies
\begin{equation} \nonumber
\begin{split}
  \int_K (\varphi \circ \Phi^{-1})(k^{-1}g_1kg_2)dk = \int_K \varphi((g_1,e)(k,k)(g_2,e))dk \\
    = \varphi((g_1,e))\varphi((g_2,e)) = (\varphi \circ \Phi^{-1})(g_1)(\varphi \circ \Phi^{-1})(g_2)
\end{split}
\end{equation}
for all $g_1,g_2 \in G$.
\end{proof}
\begin{rmk} \label{rmk:isometry}
For a compact group $G$ with compact subgroup $K$ such that $(G,K)$ is a strong Gelfand pair, the map $\Phi^*$ extends to a bijective isometry from $L^2(G // K)$ onto $L^2(\Delta K \backslash G \times K / \Delta K)$. The fact that $\Phi^*$ is isometric on $C(G // K)$ follows by elementary computation.
\end{rmk}
Let $(G,K)$ be a compact strong Gelfand pair, i.e., the group $G$ is compact and $(G,K)$ is a strong Gelfand pair. In particular, $(G,K)$ is a Gelfand pair. Let $X=G/K$ denote the corresponding homogeneous space. For an irreducible unitary representation $\pi$ of $G$, let $\mathcal{H}_{\pi}$, $\mathcal{H}_{\pi_e}$, $P_{\pi}$ and $\hat{G}_K$ be as in Section \ref{subsec:gelfandpairs}. Then $L^2(X)=\oplus_{\pi \in \hat{G}_K} \mathcal{H}_{\pi}$ (see Section \ref{subsec:nclps}). Let $h_{\pi}$ denote the spherical function corresponding to the equivalence class $\pi$ of representations. Then for every $\varphi \in L^2(K \backslash G \slash K)$ we have $\varphi=\sum_{\pi \in \hat{G}_K} c_{\pi} \dim{\mathcal{H}_{\pi}} h_{\pi}$, where $c_{\pi}=\langle \varphi,h_{\pi} \rangle$.

Recall that any unitary irreducible representation of a product of compact Lie groups arises as the tensor product of unitary irreducible representations of these groups. Also, it was already known from \cite{goldrichwigner} that a pair $(G,K)$ consisting of a locally compact group and a compact subgroup $K$ of $G$ is a strong Gelfand pair if and only if for every unitary irreducible representation $\pi$ of $G$, the space $\mathrm{Hom}_K(\pi,\tau)$ is at most one-dimensional for all unitary irreducible representations $\tau$ of $K$. Combining this with Theorem \ref{thm:bijectivecorrespondencesphericalsspherical} and Remark \ref{rmk:isometry}, the following result follows.
\begin{thm}
  Let $(G,K)$ be a compact strong Gelfand pair, and let $f \in L^2(G // K)$. Then
\[
  f = \sum_{\pi \in \widehat{G \times K}_{\Delta K}} c_{\pi} \dim{\mathcal{H}_{\pi}} (h_{\pi} \circ \Phi^{-1}) = \sum_{\pi \in \hat{G}} c_{\pi} \dim{\mathcal{H}_{\pi}} h^s_{\pi},
\]
where $h^s_{\pi}$ denotes the s-spherical function associated with $\pi$.
\end{thm}

\end{appendix}

\section*{Acknowledgements}
Appendix \ref{sec:sgp} discusses certain observations of Thomas Danielsen and the second named author. We thank Thomas Danielsen for permitting us to include them in this article. We thank Magdalena Musat, Mikael de la Salle and Henrik Schlichtkrull for numerous valuable suggestions and remarks.

\end{document}